\documentclass[11pt]{article}

\usepackage{amsfonts}
\usepackage{amscd}
\usepackage{amssymb}
\usepackage{amsthm}
\usepackage{amsmath, xspace}
\usepackage{blkarray}
\usepackage{cancel}
\usepackage{color}
\usepackage{graphics}
\usepackage{graphicx}
\usepackage{enumitem}
\usepackage{fancyhdr}
\usepackage{mathdots}
\usepackage{mathrsfs}
\usepackage{multicol}
\usepackage{stmaryrd}
\usepackage{ytableau}
\usepackage[all]{xy}

\usepackage[plainpages,backref]{hyperref}

\usepackage{lscape}

\theoremstyle{plain}
\newtheorem{thm}{Theorem}[section]

\newtheorem{prop}[thm]{Proposition}
\newtheorem{cor}[thm]{Corollary}
\theoremstyle{definition}
	
\newtheorem{remark}[thm]{Remark}
\newtheorem{example}{Example}[section]

\theoremstyle{remark}

\numberwithin{equation}{section}

\providecommand{\keywords}[1]{\textbf{\textit{Key words---}} #1}

\def\CC{\mathbb{C}}

\def\ZZ{\mathbb{Z}}

\def\ev{\mathrm{ev}}

\def\wt{\mathrm{wt}}

\usepackage{tikz}
	\usepgflibrary[patterns] 
	\usetikzlibrary{patterns} 
	\usepgflibrary{shapes.geometric}
	
\tikzstyle{D}=[draw, fill =black, circle, inner sep=0pt, minimum size=.5pt]
\DeclareMathSymbol{\sm}{\mathbin}{AMSa}{"39}

\newcounter{r}
\newcounter{s}

\newcommand\Part[1]{
        \setcounter{r}{1}
	 \foreach \x in {#1}{
 	{\ifnum\value{r}=1
		\draw (0,\value{r}-1)--(\x,\value{r}-1); 
		\fi}
	\draw (0,\value{r}) to (\x,\value{r});
   	\foreach \y in {0, ..., \x} {\draw (\y,\value{r})--(\y,\value{r}-1);}
	\addtocounter{r}{1}
 }}
 
\newcommand\Tableau[1]{
        \foreach \x [count = \c from 1] in {#1} {
		\foreach \y [count = \d from 1] in \x{
			\node at (\d-.5,\c-.5) {\scriptsize$\y$}; 
			\draw (\d,\c) to (\d,\c-1);
			{\ifnum\d=1
				\draw (0,\c) to (0,\c-1);
				\fi}
			\setcounter{r}{\d}
		}
		{\ifnum\c=1
			\draw (0,0)--(\value{r},0);
			\fi}
		\draw(0,\c) to (\value{r},\c);
		\setcounter{s}{\c}}}
		
\newcommand\sTableau[1]{
        \foreach \x [count = \c from 1] in {#1} {
		\foreach \y [count = \d from 1] in \x{
			\node at (\d-.5,\c-.5) {\tiny$\y$}; 
			\draw (\d,\c) to (\d,\c-1);
			{\ifnum\d=1
				\draw (0,\c) to (0,\c-1);
				\fi}
			\setcounter{r}{\d}
		}
		{\ifnum\c=1
			\draw (0,0)--(\value{r},0);
			\fi}
		\draw(0,\c) to (\value{r},\c);
		\setcounter{s}{\c}}}

\makeatletter
\renewcommand{\@makefnmark}{\mbox{\textsuperscript{}}}
\makeatother

\title{Monk rules for type $GL_n$ Macdonald polynomials}
\author{
Tom Halverson \quad\ \ email:\ halverson@macalester.edu \\
Arun Ram\quad\ \ email:\ aram@unimelb.edu.au \\
\\
}
\date{}

\usetikzlibrary{arrows.meta}

\begin{document}

\maketitle

\vspace{-3em}
\begin{center}
{\sl In memory of Georgia Benkart}
\end{center}


\begin{abstract}
\noindent
In this paper we give Monk rules for Macdonald polynomials which are
analogous to the Monk rules for Schubert polynomials.
These formulas are similar to the formulas given by Baratta \cite{Ba08},
but our method of derivation is to use Cherednik's interwiners.  Deriving
Monk rules by this technique addresses the relationship
between the work of Baratta and the product formulas of Yip \cite{Yi10}.
Specializations of the Monk formula's at $q=0$ and/or $t=0$ provide Monk
rules for Iwahori-spherical polynomials and for finite and affine key polynomials.
\end{abstract}


\setcounter{section}{-1}

\section{Introduction}

In this paper, we use the term \emph{electronic Macdonald polynomials} for what are commonly called
`nonsymmetric' Macdonald polynomials in the literature (see \cite{CR22} for motivation for this terminology).
The (type $GL_n$) electronic Macdonald polynomials $\{ E_\mu \ |\ \mu\in \ZZ^n\}$
form a $\CC$-basis for the ring $\CC[x_1^{\pm1}, \ldots, x_n^{\pm1}]$.
The goal of this paper is to give Monk type rules for the products
$$x_j E_\mu
\qquad\hbox{and}\qquad
(x_1+\cdots+x_j)E_\mu
\qquad\hbox{and}\qquad
E_{\varepsilon_j}E_\mu,
$$
$$x^{-1}_j E_\mu
\qquad\hbox{and}\qquad
(x^{-1}_j+\cdots+x^{-1}_n)E_\mu
\qquad\hbox{and}\qquad
E_{-\varepsilon_j}E_\mu,
$$
expanded in terms of electronic Macdonald polynomials
(here $\varepsilon_j = (0,\ldots, 0, 1, 0, \ldots, 0)$ is the $n$-tuple with $1$ in the $j$th
entry and all other entries 0).
We derive our formulas by viewing  multiplication by 
$x_j$,  multiplication by $(x_1+\cdots+x_j)$, multiplication by $E_{\varepsilon_j}$ etc.
as operators on the ring $\CC[x_1^{\pm}, \ldots, x_n^{\pm1}]$. 
Expanding a product like $x_jE_\mu$ in terms of electronic Macdonald polynomials is equivalent to writing
the operator of multiplication by $x_j$ in terms of intertwiners and Cherednik-Dunkl operators.
The expression of the operator $x_j$ in terms of intertwiners and Cherednik-Dunkl operators
can be viewed as a \emph{universal formula} for multiplication by $x_j$ in the basis of
electronic Macdonald polynomials.  These universal formulas are given in Theorem \ref{Monkops}.

To obtain the explicit expansions of the products above it is then necessary to (carefully)
``evaluate'' the universal formula at $\mu$.  These explicit expansions of the products are given in
Theorem \ref{Monkthm}.

Letting $e_1 = x_1+\cdots+x_n$ and $e_{n-1} = x_1\cdots x_n(x^{-1}_1+\cdots +x^{-1}_n)$ be the 
first and $(n-1)$st elementary symmetric functions,
Baratta \cite[Prop.\ 7, Prop.\ 8]{Ba08} gives formulas for
\begin{align*}
x_jE_\mu,
\qquad
e_1 E_\mu
\qquad\hbox{and}\qquad
e_{n-1} E_\mu
\end{align*}
expanded in terms of electronic Macdonald polynomials.  
Baratta indicates that the formula for $x_jE_\mu$ also appears in Lascoux  \cite{La08}.
The formulas of Baratta must be the same as ours, although unwinding
and comparing the notations is not immediate (at least for us).
In \cite{Ba10} Baratta computes the products
$e_r E_\mu$ where $e_r$ denotes the $r$th elementary symmetric function.
It might be possible to give alternate derivations
of the products $e_r E_\mu$ for general $r$ using the methods of this paper.  

Baratta's approach is similar to that of Lascoux, using the interpolation Macdonald polynomials
and results of Knop and Sahi \cite{Kn96} and \cite{Sa96}.
Our approach uses the intertwiners and their relation with the Cherednik-Dunkl operators.
Computing these formulas via intertwiners addresses the relationship between the formulas 
of Baratta and the methods of Yip \cite{Yi10},
who gives some related expansions, but in an alcove walk form.

The motivation for the term ``Monk rules'' comes from Schubert calculus.
In Schubert calculus, Monk's rules for the Schubert polynomials $\mathfrak{S}_w$ are
$$x_j \mathfrak{S}_w 
= \Big(\!\!\!\!\sum_{1\le i<j\atop \ell(ws_{ij})=\ell(w)+1} \mathfrak{S}_{ws_{ij}}\Big)
- \Big(\!\!\!\!\sum_{j<i\le n\atop \ell(ws_{ji})=\ell(w)+1} \mathfrak{S}_{ws_{ji}}\Big)
\qquad\hbox{and}\qquad
\mathfrak{S}_{s_{r,r+1}} \mathfrak{S}_w 
= \sum_{i\le r< j\atop \ell(ws_{ij})= \ell(w)+1} \mathfrak{S}_{ws_{ij}}.
$$
(here $w$ is a permutation in the
symmetric group $S_n$ and $s_{ij}$ denotes the transposition which switches $i$ and $j$).
These rules are proved in \cite[(4.15),($4.15'$),($4.15''$)]{Mac91}.
A compendium of similar formulas for type $GL_n$ Grothendieck polynomials is given in 
\cite[\S1,2, \S1.3]{LS04}.  Though the analogies are tantalizing, we have not, in any generality, 
made a concrete connection between our `Monk formulas' for Macdonald poylnomials 
and the formulas which appear in Schubert calculus.  Part (c) of Corollary \ref{spclzdMonk}
provides a different formulation and proof of \cite[Theorem 3.3.6]{AQ19}.  
Other formulas related to Corollary \ref{spclzdMonk} appear in Assaf \cite{As21}
and Gibson \cite{Gib19}
who, respectively, use the combinatorics of Kohnert diagrams and monomial crystals.

\smallskip\noindent
\textbf{Acknowledgements.}  We are very grateful to Zajj Daugherty for tikzing the picture in \eqref{rotmuCpic}.

\section{Macdonald polynomials}

Let $n\in \ZZ_{>0}$.
The (Laurent) polynomial ring $\CC[x_1^{\pm1}, \ldots, x_n^{\pm1}]$ has basis
$$\{x^\mu\ |\ \mu\in \ZZ^n\},
\qquad\hbox{where}\quad
x^\mu = x_1^{\mu_1}\cdots x_n^{\mu_n}
\quad\hbox{for $\mu = (\mu_1, \ldots, \mu_n)\in \ZZ^n$.}
$$
The symmetric group
$S_n$ acts on $\CC[x_1^{\pm1}, \ldots, x_n^{\pm1}]$ 
by permuting the variables $x_1, \ldots, x_n$.
The symmetric group $S_n$ acts on $\ZZ^n$ by permuting the positions of the entries.  The two actions
are related by $wx^\mu = x^{w\mu}$ for $w\in S_n$ and $\mu\in \ZZ^n$.

Let $q, t^{\frac12}\in \CC^\times$.   For $j\in \{1, \dots, n\}$ let $X_j$ be the operator
on $\CC[x_1^{\pm1}, \ldots, x_n^{\pm1}]$ given by multiplication by $x_j$.
For $i\in \{1, \ldots, n-1\}$ let
$s_i\in S_n$ be the transposition which switches $i$ and $i+1$.  For $j\in \{1, \ldots, n\}$ 
let $y_j$ be the operator on  $\CC[x_1^{\pm1}, \ldots, x_n^{\pm1}]$ which replaces each 
occurrence of $x_j$ with $q^{-1}x_j$.  In formulas, if $f\in \CC[x_1^{\pm1}, \ldots, x_n^{\pm1}]$ then
\begin{align}
(X_jf)(x_1, \ldots, x_n) &= x_j \cdot f(x_1, \ldots, x_n), 
\nonumber  \\
(s_if)(x_1, \ldots, x_n) &= f(x_1, \ldots, x_{i-1}, x_{i+1}, x_i, x_{i+1}, \ldots, x_n), 
\label{Xjsiyjops}
\\
(y_jf)(x_1, \ldots, x_n) &= f(x_1, \ldots, x_{j-1}, q^{-1}x_j, x_{j+1}, \ldots, x_n).
\nonumber
\end{align}

Define operators $T_1, \ldots, T_{n-1}$, $T_\pi$ and $T^\vee_\pi$ on 
$\CC[x_1^{\pm1}, \ldots, x_n^{\pm1}]$ by 
\begin{equation}
T_i = -t^{-\frac12} + t^{-\frac12}(1+s_i)\frac{1-tx_i^{-1}x_{i+1}}{1-x^{-1}_ix_{i+1}},
\qquad
T_\pi = s_1s_2\cdots s_{n-1} y_n, \qquad
T^\vee_\pi = X_1T_1\cdots T_{n-1}.
\label{DAHAonCX}
\end{equation}
The \emph{Cherednik-Dunkl operators} are 
\begin{equation}
Y_1 = T_\pi T_{n-1}\cdots T_1,
\quad Y_2 = T_1^{-1}Y_1T_1^{-1}, \quad Y_3 = T_2^{-1}Y_2T_2^{-1}, \quad \ldots,\quad
Y_n= T_{n-1}^{-1}Y_{n-1}T_{n-1}^{-1}.
\label{GLnCDopsdefn}
\end{equation}

Given $\mu\in \ZZ^n$ let $v_\mu\in S_n$ be the minimal length permutation which rearranges
$\mu$ into weakly increasing order.  Explicitly, the permutation $v_\mu$ is given by
\begin{equation}
v_\mu(r) = 1+
\#\{r'\in \{1, \ldots, r-1\} \ |\ \mu_{r'}\le \mu_r \}  
+ \#\{r' \in \{r+1, \ldots, n\} \ |\ \mu_{r'}< \mu_r \}.
\label{vmuformula}
\end{equation}
By definition, the electronic Macdonald polynomials $E_\mu$ are the simultaneous 
eigenvectors for the action of the Cherednik-Dunkl operators,
\begin{equation}
Y_iE_\mu = q^{-\mu_i}t^{-v_\mu(i)}t^{\frac12(n+1)} E_\mu.
\label{eigenvalue}
\end{equation}
The ``evaluate at $\mu$'' homomorphism $\ev^t_\mu\colon \CC[Y^{\pm1}_1, \ldots, Y^{\pm1}_n]\to \CC$ is given by
\begin{equation}
\ev^t_\mu(Y_i) = q^{-\mu_i}t^{-(v_\mu(i)-1)+\frac12(n-1)}
\label{evtmudefn}
\end{equation}
(so that $\ev^t_\mu$ specializes $Y_i$ to the value 
$q^{-\mu_i}t^{-(v_\mu(i)-1)+\frac12(n-1)}$).  Extend $\ev^t_\mu$ to those elements of the 
field $\CC(Y_1, \ldots, Y_n)$ for which the specialized denominator does not vanish.
By \eqref{eigenvalue}, if $f(Y)\in \CC(Y_1,\ldots, Y_n)$ and $\ev^t_\mu(f)$ is defined then
\begin{equation}
f(Y) E_\mu = \ev^t_\mu(f(Y)) E_\mu.
\label{evequalsev}
\end{equation}

The \emph{interwiners} are
\begin{equation}
\tau_\pi^\vee = T^\vee_\pi,
\quad\hbox{and}\quad
\tau_i^\vee = T_i + \frac{t^{-\frac12}(1-t)}{1-Y^{-1}_i Y_{i+1}}\quad\hbox{for $i\in \{1, \ldots, n-1\}$.}
\label{earlyintertwinerdefn}
\end{equation}
Using the definition of $\tau^\vee_i$ and the relation $T_i-T^{-1}_i = t^{\frac12}-t^{-\frac12}$,
\begin{equation}
\tau^\vee_i= T_i+f_{i+1,i}^+ = T^{-1}_i + f_{i+1,i}^-,
\label{tauiexp}
\end{equation}
where
\begin{equation}
\qquad\hbox{where}\quad
f_{ij}^+ = \frac{t^{-\frac12}(1-t)}{1-Y_iY^{-1}_j}
\quad\hbox{and}\quad f_{ij}^- = \frac{t^{-\frac12}(1-t)Y_iY^{-1}_j}{1-Y_iY^{-1}_j},
\label{fijpmdefns}
\end{equation}
for $i,j\in \{1, \ldots, n\}$ with $i\ne j$.
The following key relations are proved (for example) in \cite[Prop.\ 5.5]{GR21},
\begin{equation}
Y_1\tau^\vee_\pi = q^{-1}\tau^\vee_\pi Y_n
\quad\hbox{and}\quad Y_i\tau^\vee_\pi = \tau^\vee_\pi Y_{i-1}
\quad\hbox{for $i\in \{2, \ldots, n\}$, and}
\label{intrelsA}
\end{equation}
\begin{equation}
Y_i \tau^\vee_i = \tau^\vee_i Y_{i+1},
\qquad
Y_{i+1} \tau^\vee_i = \tau^\vee_i Y_i,
\qquad\hbox{and}\qquad
Y_k \tau^\vee_i = \tau^\vee_i Y_k,
\label{intrelsB}
\end{equation}
for $i\in \{1, \ldots, n-1\}$ and $k\in \{1, \ldots, n\}$ with $k\not\in \{i,i+1\}$.

The following Proposition gives an explicit expression for $E_\mu$
as a sequence of intertwiners acting on the polynomial 1.

\begin{prop} \emph{\cite[Proposition 5.7 and Proposition 2.2(a)]{GR21}}
\item[(a)] Let $\mu = (\mu_1, \ldots, \mu_n)\in \ZZ^n_{\ge 0}$ and write $(r,c)\in \mu$ if
$r\in \{1, \ldots, n\}$ and $c\in \{1, \ldots, \mu_r\}$.   For $(r,c)\in \mu$ define
$$u_\mu(r,c) = \#\{ r'\in \{1, \ldots, r-1\}\ |\ \mu_{r'}<c\le \mu_r \} 
+\#\{ r'\in \{r+1, \ldots, n\}\ |\ \mu_{r'} < c-1 < \mu_r \}.
$$
Then
$$E_\mu = t^{-\frac12\ell(v^{-1}_\mu)} 
\Big(\prod_{r=1}^n \prod_{c=1}^{\mu_r}
(\tau^\vee_{u_\mu(r,c)} \cdots \tau^\vee_2\tau^\vee_1\tau^\vee_\pi)\Big)\cdot 1.$$
where the product is taken in order defined by
$(r_1,c_1)< (r_2,c_2)$ if $c_1<c_2$, and $(r_1,c)<(r_2,c)$ if $r_1<r_2$.
\item[(b)] If $\mu = (\mu_1, \ldots, \mu_n)\in \ZZ^n$ has a negative entry then
$$E_\mu = (x_1\cdots x_n)^{-m} E_{(\mu_1+m, \ldots, \mu_n+m)},
\qquad \hbox{where $-m$ is the most negative entry of $\mu$.} 
$$
\end{prop}

\begin{prop}  \label{Eepsi}
For $i\in \{1, \dots, n\}$ let $\varepsilon_i = (0,\ldots, 0, 1, 0, \ldots, 0)\in \ZZ^n$ with
$1$ in the $i$th entry and all other entries 0.  
Then
\begin{align*}
E_{\varepsilon_i} &= x_i+ \frac{(1-t)}{(1-qt^{n-i+1})}(x_{i-1}+\cdots +x_1)
\quad\hbox{and} \\
E_{-\varepsilon_i} &= x^{-1}_i+ \Big(\frac{1-t}{1-qt^i}\Big)(x^{-1}_{i+1}+\cdots +x^{-1}_n).
\end{align*}
\end{prop}
\begin{proof}
Since $v_{-\varepsilon_{i+1}}= s_1\cdots s_i$ then
$$Y^{-1}_iY_{i+1} E_{-\varepsilon_{i+1}} 
= q^0t^{v_{-\varepsilon_{i+1}}(i)} q^1 t^{-v_{-\varepsilon_{i+1}}(i+1)}
E_{-\varepsilon_{i+1}} 
= q t^{i+1} t^{-1} E_{-\varepsilon_{i+1}} 
= q t^i E_{-\varepsilon_{i+1}} 
$$
The base case is $E_{-\varepsilon_n} = x_n^{-1}$ and the induction step is
\begin{align*}
E_{-\varepsilon_i} 
&= t^{\frac12} \tau^\vee_i E_{-\varepsilon_{i+1}} 
=\Big(t^{\frac12}T_i+\frac{(1-t)}{1-Y^{-1}_iY_{i+1}}\Big)E_{-\varepsilon_{i+1}}
=\Big(t^{\frac12}T_i+\frac{(1-t)}{1-qt^i}\Big)E_{-\varepsilon_{i+1}}
\\
&=\Big(t^{\frac12}T_i+\frac{(1-t)}{1-qt^i}\Big)(x^{-1}_{i+1}
+ \Big(\frac{1-t}{1-qt^{i+1}}\Big)
(x^{-1}_{i+2}+\cdots+x^{-1}_n)
\\
&=x^{-1}_i
+\Big(\frac{1-t}{1-qt^i}\Big)x^{-1}_{i+1}
+ \Big(\frac{1-t}{1-qt^{i+1}}\Big)\Big(t+\frac{1-t}{1-qt^i }\Big)
(x^{-1}_{i+2}+\cdots+x^{-1}_n)
\\
&=x^{-1}_i
+\Big(\frac{1-t}{1-qt^i}\Big)(x^{-1}_{i+1}+\cdots+x^{-1}_n).
\end{align*}
The proof of the first statement is similar (see \cite[Prop.\ 3.5]{GR21} for details).
\end{proof}

\begin{remark} \textbf{The source of the statistics $v_\mu(r)$ and $u_\mu(r,c)$.}
The minimal length permutation
which rearranges $\mu$ into weakly increasing order is
$v_\mu = (v_\mu(1), \ldots, v_\mu(n))$.  The affine Weyl group for type $GL_n$
is the group of $n$-periodic permutations.  If $t_\mu$ denotes the $n$-periodic permutation
which is the translation in $\mu$ then $t_\mu = u_\mu v_\mu$ with $\ell(t_\mu) = \ell(u_\mu)+\ell(v_\mu)$ and $u_\mu$ has a reduced word
$$u_\mu = \prod_{(r,c)\in \mu} (s_{u_\mu(r,c)}\cdots s_2s_1 \pi),$$
where $s_i\in S_n$ is the transposition which switches $i$ and $i+1$ and $\pi$ is the 
$n$-periodic permutation given by $\pi(i)=i+1$.  See \cite[\S2 and Prop.\ 2.2(a)]{GR21}.
\qed
\end{remark}

\section{Operator expansions}

Let $j\in \{1, \ldots, n\}$ and let $C \subseteq \{1, \ldots, n\}$.
Writing $C= \{a_1, \ldots, a_n\}$ with $a_1<\cdots < a_m$ define
\begin{equation}
f_C(Y) = \frac{t^{-(m-1)/2}}{1-qY_{a_1}Y^{-1}_{a_m}}\Big(
\prod_{i=1}^{m-1} \frac{1-t}{1-Y_{a_i}Y^{-1}_{a_{i+1}}}\Big).
\label{fCdefn}
\end{equation}
Then define
\begin{equation}
F_{C,j}(Y) = \begin{cases} 0, &\hbox{if $j\not\in C$,}
\\
1-qY_{a_1}Y^{-1}_{a_m}, &\hbox{if $j=a_p$ and $p=1$,} \\
Y_{a_1}Y^{-1}_{a_p} - Y_{a_1}Y^{-1}_{a_{p-1}}, &\hbox{if $j=a_p$ and $p\ne 1$,}
\end{cases}
\label{FCjYdefn}
\end{equation}
\begin{equation}
A_{C,j}(Y) = 
\begin{cases}
0, 
&\hbox{if $j<a_1$,}
\\
Y_{a_1}Y^{-1}_{a_p} - qY_{a_1}Y^{-1}_{a_m},
&\hbox{if $a_p\le j< a_{p+1}$,}
\\
(1-q)Y_{a_1}Y^{-1}_{a_m},
&\hbox{if $j>a_m$,}
\end{cases}
\label{ACjYdefn}
\end{equation}
\begin{equation}
\Phi_{C,j}(Y) = \begin{cases} 0, &\hbox{if $j\not\in C$,} \\
1-qY_{a_1}Y^{-1}_{a_m}, &\hbox{if $j=a_p$ and $p=m$,} \\
Y_{a_1}Y^{-1}_{a_p} - Y_{a_1}Y^{-1}_{a_{p-1}}, &\hbox{if $j=a_p$ and $p\ne m$,}
\end{cases}
\label{PhiCjYdefn}
\end{equation}
\begin{equation}
\Psi_{C,j}(Y) =
\begin{cases}
0, 
&\hbox{if $j\ge a_m$,}
\\
Y_{a_p}Y^{-1}_{a_m} - qY_{a_1}Y^{-1}_{a_m},
&\hbox{if $a_{p-1}< j\le a_p$,}
\\
(1-q)Y_{a_1}Y^{-1}_{a_m},
&\hbox{if $j\le a_1$,}
\end{cases}
\label{PsiCjYdefn}
\end{equation}
and
\begin{equation}
\begin{array}{l}
\displaystyle{B_{C,j}(Y) = F_{C,j}(Y)+\frac{1-t}{1-qt^{n-j+1}} A_{C,j}(Y)}
\quad\hbox{and} \\
\\
\displaystyle{ \Omega_{C,j}(Y) = \Phi_{C,j}(Y)+ \Big(\frac{1-t}{1-qt^j}\Big) \Psi_{C,j+1}(Y).}
\end{array}
\label{BOmegadefn}
\end{equation}
Write the complement of $C$ in $\{1, \ldots n\}$ as
$$C^c = \{b_1, \ldots, b_{n-m}\}\quad\hbox{with}\quad
b_1<\cdots < b_r < j< b_{r+1}< \cdots < b_{n-m},
$$
and define
\begin{equation}
\begin{array}{l}
\tau_{C,j} = \tau^\vee_{b_r}\tau^\vee_{b_{r-1}}\cdots \tau^\vee_{b_1}
\tau^\vee_\pi \tau^\vee_{b_{n-m}-1}\cdots \tau^\vee_{b_{r+1}-1}
\quad\hbox{and} \\
\\
\rho_{C,j} = \tau^\vee_{b_{r+1}-1}\cdots \tau^\vee_{b_{n-m}-1}(\tau^\vee_\pi)^{-1}
\tau^\vee_{b_1}\cdots \tau^\vee_{b_r},
\end{array}
\label{tauCrhoCdefn}
\end{equation}
where the $\tau^\vee_i$ are as in \eqref{earlyintertwinerdefn}.

\begin{example}  
\textbf{Examples of $f_C(Y)$, $F_{C,j}(Y)$ and $\tau^\vee_{C,j}$.}
Let $n=11$ and $C = \{2,5,7,9,10\}$.  Then
$$f_C(Y) = \frac{1}{t^{-\frac12}(1-t)}  f^+_{2,10+K} f^+_{25} f^+_{57}f^+_{79}f^+_{9,10},$$
where $f_{ij}^+$ is as in \eqref{fijpmdefns}.  Then
\begin{align*}
F_{C,2}(Y) &= 1-qY_2Y^{-1}_{10},
&F_{C,5}(Y) &= Y_2Y^{-1}_5-1,
&F_{C,7}(Y) &= Y_2Y^{-1}_7-Y_2Y^{-1}_5,
\\
F_{C,9}(Y) &=Y_2Y^{-1}_9 - Y_2Y^{-1}_7
&F_{C,10}(Y) &= Y_2Y^{-1}_{10}-Y_2Y^{-1}_9,
\end{align*}
and
$$\tau^\vee_{C,7} 
= \tau^\vee_6\tau^\vee_4\tau^\vee_3\tau^\vee_1 
\tau^\vee_\pi \tau^\vee_{11-1} \tau^\vee_{8-1}
= \tau^\vee_6\tau^\vee_4\tau^\vee_3\tau^\vee_1 
\tau^\vee_\pi \tau^\vee_{10} \tau^\vee_7,
\qquad\hbox{since\quad $C^c=\{1,3,4,6,8,11\}$.}
$$
\qed
\end{example}

\begin{thm} \label{Monkops}
\textbf{(Monk rules: operator form)}
Let $j\in \{1, \ldots, n\}$. As in \eqref{Xjsiyjops}, let $X_j$ denote the operator on $\CC[x_1^{\pm1}, \ldots, x_n^{\pm1}]$
given by multiplication by $x_i$ and let $E_{\varepsilon_j}$ and $E_{-\varepsilon_j}$ be the Macdonald polynomials
of Proposition \ref{Eepsi}, identified with the operators on $\CC[x_1^{\pm1}, \ldots, x_n^{\pm1}]$
given by multiplication by $E_{\varepsilon_j}$ and $E_{-\varepsilon_j}$, respectively.
Use the notations of \eqref{fCdefn}-\eqref{tauCrhoCdefn}.
Then, as operators on $\CC[x_1^{\pm1}, \ldots, x_n^{\pm1}]$,
\item[]\qquad \emph{(a)}
$\displaystyle{
X_j  = \sum_{C\subseteq \{1, \ldots, n\}\atop C \cap \{j\}\ne \emptyset} 
\tau_{C,j} F_{C,j}(Y)f_C(Y),
}
$
\item[]\qquad \emph{(b)}
$\displaystyle{
X_1+\cdots + X_j  = \sum_{C\subseteq \{1, \ldots, n\}\atop C \cap \{1,\ldots, j\}\ne \emptyset} 
\tau_{C,j} A_{C,j}(Y)f_C(Y),
}
$
\item[]\qquad \emph{(c)}
$\displaystyle{
E_{\varepsilon_j}  = \sum_{C\subseteq \{1, \ldots, n\}\atop C \cap \{1,\ldots,j\}\ne \emptyset} 
\tau_{C,j} B_{C,j}(Y)f_C(Y).
}
$
\item[]\qquad \emph{(d)}
$\displaystyle{
X^{-1}_j  = \sum_{C\subseteq \{1, \ldots, n\}\atop D \cap \{j\}\ne \emptyset} 
\rho_{C,j} \Phi_{C,j}(Y)f_C(Y),
}
$
\item[]\qquad \emph{(e)}
$\displaystyle{
X^{-1}_j+\cdots + X^{-1}_n  = \sum_{C\subseteq \{1, \ldots, n\}\atop C \cap \{j,\ldots, n\}\ne \emptyset} 
\rho_{C,j} \Psi_{C,j}(Y)f_C(Y)
}
$
\item[]\qquad \emph{(f)}
$\displaystyle{
E_{-\varepsilon_j}  = \sum_{C\subseteq \{1, \ldots, n\}\atop C \cap \{j,\ldots,n\}\ne \emptyset} 
\rho_{C,j} \Omega_{C,j}(Y)f_C(Y).
}
$
\end{thm}
\begin{proof}  Since the proof of (a) is longer, let us first make remarks about the proofs of (b)-(f).

\smallskip\noindent
(b) This follows from (a) and the observation that $A_{C,j} = F_{C,1}+\cdots +F_{C,j}$.

\smallskip\noindent
(c) By the first identity in Proposition \ref{Eepsi},
$$E_{\varepsilon_j} = x_j + \frac{1-t}{1-qt^{n-j+1}}(x_{j-1}+\cdots + x_1),
\qquad\hbox{so that}\quad
B_{C,j} = F_{C,j}+ \frac{1-t}{1-qt^{n-j+1}} A_{C,{j-1}}.
$$

\smallskip\noindent
(d) The proof is analogous to the proof of (a) by expanding
\begin{align*}
X^{-1}_j &= T_j\cdots T_{n-1} (\tau^\vee_\pi)^{-1} T^{-1}_1\cdots T^{-1}_{j-1}
\\
&=(\tau^\vee_j - f^+_{j+1,j}) \cdots
(\tau^\vee_{n-1}  - f^+_{n,n-1}) 
(\tau^\vee_\pi)^{-1}
(\tau^\vee_1 - f^-_{2,1}) \cdots
(\tau^\vee_{j-1}  - f^-_{j,j-1})
\\
&= \sum_{L\subseteq\{j,\ldots, n-1\} \atop R\subseteq\{1,\ldots, j-1\} } 
(-1)^{\vert L\vert+\vert R\vert} \varpi_L \tau^\vee_\pi \varpi_R,
\end{align*}
$$\varpi_L = \varpi_j\cdots \varpi_{n-1}
\quad\hbox{with}\quad
\varpi_i = \begin{cases} 
\tau^\vee_i, &\hbox{if $i\in \{j, \ldots, n-1\}$ and $i\not\in L$,} \\
f^+_{i+1,i}, &\hbox{if $i\in\{j, \ldots, n-1\}$ and $i\in L$,} 
\end{cases}\qquad\hbox{and}
$$
$$\varpi_R = \varpi_1\cdots \varpi_{j-1}
\quad\hbox{with}\quad
\varpi_i = \begin{cases}
\tau^\vee_i, &\hbox{if $i\in \{1, \ldots, j-1\}$ and $i\not\in R$,} \\
f^-_{i+1,i}, &\hbox{if $i\in\{1, \ldots, j-1\}$ and $i\in R$,} 
\end{cases}
\qquad\ \ \ 
$$
An example is provided in Example \ref{invproofexample}.

\smallskip\noindent
(e) This follows from (d) and the observation that $\Psi_{D,j} = \Phi_{D,j}+\cdots +\Phi_{D,n}$.

\smallskip\noindent
(f) By the second identity in Proposition \ref{Eepsi}
$$E_{-\varepsilon_j} = x^{-1}_j + \Big(\frac{1-t}{1-qt^j}\Big) (x^{-1}_{j+1}+\cdots + x^{-1}_n),
\qquad\hbox{so that}\quad
\Omega_{C,j} = \Phi_{D,j}+ \Big(\frac{1-t}{1-qt^j}\Big) \Psi_{D,j+1}.
$$

\smallskip\noindent
(a) As in \eqref{Xjsiyjops}, let $X_j$ denote the operator on $\CC[x^{\pm}_1, \ldots, x^{\pm1}_n]$ given by mutliplication by $x_j$.  The operator $X_j$ can
be written in terms of $\tau^\vee_\pi$ and $T_1, \ldots, T_{n-1}$ in the form
$$X_j = T_{j-1}\cdots T_1 \tau^\vee_\pi T^{-1}_{n-1}\cdots T^{-1}_j$$
(see, for example, \cite[(5.3) and \S5.6 and \S5.7]{GR21}).
Then using \eqref{tauiexp} gives
\begin{align}
X_j &= T_{j-1}\cdots T_1 \tau^\vee_\pi T^{-1}_{n-1}\cdots T^{-1}_j
\nonumber \\
&=(\tau^\vee_{j-1} - f^+_{j,j-1}) \cdots
(\tau^\vee_1  - f^+_{2,1}) 
\tau^\vee_\pi
(\tau^\vee_{n-1} - f^-_{n,n-1}) \cdots
(\tau^\vee_j  - f^-_{j+1,j})
\nonumber \\
&= \sum_{L\subseteq\{j-1,\ldots, 1\} \atop R\subseteq\{n-1,\ldots, j\} } 
(-1)^{\vert L\vert+\vert R\vert} \omega_L \tau^\vee_\pi \omega_R,
\label{expand}
\end{align}
where
$$\omega_L = \omega_{j-1}\cdots \omega_1
\quad\hbox{with}\quad
\omega_i = \begin{cases} 
\tau^\vee_i, &\hbox{if $i\in \{j-1, \ldots, 1\}$ and $i\not\in L$,} \\
f^+_{i+1,i}, &\hbox{if $i\in\{j-1, \ldots, 1\}$ and $i\in L$,} 
\end{cases}\qquad\hbox{and}
$$
$$\omega_R = \omega_{n-1}\cdots \omega_j
\quad\hbox{with}\quad
\omega_i = \begin{cases}
\tau^\vee_i, &\hbox{if $i\in \{n-1, \ldots, j\}$ and $i\not\in R$,} \\
f^-_{i+1,i}, &\hbox{if $i\in\{n-1, \ldots, j\}$ and $i\in R$.} 
\end{cases}
\qquad\ \ \ 
$$
Write
$$
\begin{array}{l}
L = \{ \ell_1, \ldots, \ell_a\}, \\
R = \{ r_1, \ldots, r_b\},
\end{array}
\qquad\hbox{with}\qquad
n>r_1> \ldots > r_b \ge j >\ell_1>\ldots > \ell_a > 0,$$
and use the relations \eqref{intrelsA} and \eqref{intrelsB} 
to move all $\tau^\vee_i$ in $\omega_L \tau^\vee_\pi \omega_R$
to the left so that
\begin{equation}
(-1)^{\vert L\vert+\vert R\vert} \omega_L \tau^\vee_\pi \omega_R = \tau^\vee_{L^c,R^c} f_{L,R},
\label{split}
\end{equation}
where
$$
\begin{array}{l}
\hbox{$L^c=\{k_1,\ldots, k_{j-1-a}\}$ with $k_1>\cdots>k_{j-1-a}$ is the complement of $L$ in $\{j-1,\ldots,1\}$,}
\\
\hbox{$R^c=\{q_1,\ldots, q_{n-j-b}\}$ with $q_1>\cdots>q_{n-j-b}$ is the complement of $R$ in $\{n-1,\ldots,j\}$,}
\end{array}
$$
$$
\tau^\vee_{L^r,R^c} = \tau^\vee_{k_1}\cdots \tau^\vee_{k_{j-1-a}} \tau^\vee_\pi
\tau^\vee_{q_1}\cdots \tau^\vee_{q_{n-j-b}},
$$
and
$$f_{L,R}
=(-1)^{\vert L\vert+\vert R\vert}  f^+_{\ell_1,\ell_2} f^+_{\ell_2,\ell_3}\cdots f^+_{\ell_{a-1},\ell_a}
f^+_{\ell_a,r_1+1+K} f^-_{r_1+1,r_2+1}f^-_{r_2+1,r_3+1}\cdots f^-_{r_{b-1}+1,r_b+1}
f^-_{r_b+1,j},
$$
where
$$f^+_{i,j+K} = \frac{t^{-\frac12}(1-t)}{1-qY_iY^{-1}_j}$$
(the $K$ in this expression is a formal notational symbol and has no other meaning in this context).
An example of this process of using the relations \eqref{intrelsA} and \eqref{intrelsB} 
to move all the $\tau^\vee_i$ in $\omega_L \tau^\vee_\pi \omega_R$
to the left is given in Example \ref{proofexample}.

Let $C=\{a_1,\ldots, a_m\} = \{\ell_a, \ldots, \ell_1, j, r_b+1, \ldots, r_1+1\}$
and let $C^c=\{b_1, \ldots, b_{n-m}\}$ be the complement of $C$ in $\{1, \ldots, n\}$ 
so that
\begin{align*}
&C=\{a_1,\ldots, a_m\}
\qquad\hbox{with}\qquad 
1\le a_1<\cdots<a_m\le n
\qquad\hbox{and}\qquad j\in C, \\
\hbox{and}\quad
&C^c = \{b_1,\ldots, b_{n-m}\}
\quad\hbox{with}\qquad
b_1<\cdots < b_r < j < b_{r+1} \cdots < b_{n-m}.
\end{align*}
Then
\begin{equation}
\tau^\vee_{L^c,R^c} 
= \tau^\vee_{b_r}\tau^\vee_{b_{r-1}}\cdots \tau^\vee_{b_1}
\tau^\vee_\pi \tau^\vee_{b_{n-m}-1}\cdots \tau^\vee_{b_{r+1}-1}
= \tau^\vee_{C,j},
\label{taupart}
\end{equation}
and letting $p$ be such that $j=a_p$ and rearranging the factors in $f_{L,R}$ gives
\begin{align*}
f_{L,R}
&= (-1)^{\vert L\vert+\vert R\vert} f^+_{\ell_a,r_1+1+K}
f^+_{\ell_{a-1},\ell_a}\cdots  f^+_{\ell_2,\ell_3} f^+_{\ell_1,\ell_2}
f^-_{r_b+1,j} f^-_{r_{b-1}+1,r_b+1}\cdots  f^-_{r_2+1,r_3+1}f^-_{r_1+1,r_2+1}
\\
&= (-1)^{\vert L\vert+\vert R\vert} f^+_{a_1,a_m+K}
f^+_{a_2a_1}\cdots  f^+_{a_{p-2}a_{p-3}}f^+_{a_{p-1}a_{p-2}}
f^-_{a_{p+1}a_p} f^-_{a_{p+2}a_{p+1}}\cdots  f^-_{a_{m-1}a_{m-2}}f^-_{a_ma_{m-1}}
\\
&= (-1)^{\vert L\vert+\vert R\vert} f^+_{a_1,a_m+K}
f^+_{a_2a_1}\cdots  f^+_{a_{p-2}a_{p-3}}f^+_{a_{p-1}a_{p-2}}
f^-_{a_{p+1}a_p} f^-_{a_{p+2}a_{p+1}}\cdots  f^-_{a_{m-1}a_{m-2}}f^-_{a_ma_{m-1}}.
\end{align*}

Now use
\begin{align*}
f^+_{ki} 
&= \frac{t^{-\frac12}(1-t)}{1-Y_kY^{-1}_i}
= -\frac{t^{-\frac12}(1-t)Y_iY^{-1}_k}{1-Y_iY^{-1}_k} 
= -Y_iY^{-1}_k f^+_{ik}
\qquad\hbox{and} \\
f^-_{ki} 
&= \frac{t^{-\frac12}(1-t)Y_kY^{-1}_i}{1-Y_kY^{-1}_i}
= -\frac{t^{-\frac12}(1-t)}{1-Y_iY^{-1}_k}
= -f^+_{ik}.
\end{align*}
If $p\ne 1$ then 
\begin{align*}
f_{L,R}
&= (-1)^{\vert L\vert+\vert R\vert} f^+_{a_1,a_m+K}
f^+_{a_2a_1}\cdots  f^+_{a_{p-2}a_{p-3}}f^+_{a_{p-1}a_{p-2}}
f^-_{a_{p+1}a_p} f^-_{a_{p+2}a_{p+1}}\cdots  f^-_{a_{m-1}a_{m-2}}f^-_{a_ma_{m-1}}
\\
&=(-1)^{\vert L\vert+\vert R\vert} (-1)^{\vert C\vert -2} f^+_{a_1,a_m+K}
(Y_{a_1}Y^{-1}_{a_2} f^+_{a_1a_2})\cdots  (Y_{a_{p-3}}Y^{-1}_{a_{p-2}} f^+_{a_{p-3}a_{p-2}})
(Y_{a_{p-2}}Y^{-1}_{a_{p-1}} f^+_{a_{p-2}a_{p-1}})
\\
&\qquad\qquad\qquad\qquad
\cdot f^+_{a_pa_{p+1}} f^+_{a_{p+1}a_{p+2}}\cdots  f^+_{a_{m-1}a_m}
\\
&=(-1) \frac{f^+_{a_1,a_m+K}}{f^+_{a_{p-1}a_p}} Y_{a_1}Y^{-1}_{a_{p-1}}
\Big(\prod_{i=1}^{m-1} f^+_{a_ia_{i+1}}\Big)
\\
&=(-1) Y_{a_1}Y^{-1}_{a_{p-1}} (1-Y_{a_{p-1}}Y^{-1}_{a_p}) \frac{1}{t^{-\frac12}(1-t)}
f^+_{a_1,a_m+K} 
\Big(\prod_{i=1}^{m-1} f^+_{a_ia_{i+1}}\Big)
\\
&=(Y_{a_1}Y^{-1}_{a_p}-Y_{a_1}Y^{-1}_{a_{p-1}})
\frac{t^{-\frac{m-1}{2}}(1-t)^{m-1}}{1-qY_{a_1}Y^{-1}_{a_m}} 
\Big( \prod_{i=1}^{m-1} \frac{1}{1-Y_{a_i}Y^{-1}_{a_{i+1}}}\Big)
= F_{C,j}(Y) f_C(Y),
\end{align*}
and if $p=1$ then
\begin{align*}
f_{L,R} 
&= (-1)^{\vert L\vert+\vert R\vert} f^-_{r_1+1,r_2+1}f^-_{r_2+1,r_3+1}\cdots f^-_{r_{b-1}+1,r_b+1}
f^-_{r_b+1,j}
\\
&= (-1)^{\vert L\vert+\vert R\vert} f^-_{a_m,a_{m-1}}f^-_{a_{m-1}a_{m-2}}\cdots f^-_{a_3a_2}f^-_{a_2a_1}
\\
&= f^+_{a_{m-1}a_m}f^+_{a_{m-2}a_{m-1}}\cdots f^+_{a_2a_3}f^+_{a_1a_2}
=\frac{1}{f^+_{a_1,a_m+K}}f^+_{a_1,a_m+K}\Big(\prod_{i=1}^{m-1} f^+_{a_ia_{i+1}}\Big)
\\
&= (1-qY_{a_1}Y^{-1}_{a_m}) \frac{1}{t^{-\frac12}(1-t)}
f^+_{a_1,a_m+K}\Big(\prod_{i=1}^{m-1} f^+_{a_ia_{i+1}}\Big)
\\
&= (1-qY_{a_1}Y^{-1}_{a_m}) 
\frac{t^{-\frac{m-1}{2}}(1-t)^{m-1}}{1-qY_{a_1}Y^{-1}_{a_m}} 
\Big( \prod_{i=1}^{m-1} \frac{1}{1-Y_{a_i}Y^{-1}_{a_{i+1}}}\Big)
= F_{C,j}(Y) f_C(Y).
\end{align*}
Inserting these expressions for $f_{L,R}$ and the expression for $\tau^\vee_{L^c,R^c}$
in \eqref{taupart} into \eqref{split} and \eqref{expand} gives the formula in the statement.
\end{proof}

\begin{remark} The $B_{C,j}$ defined in \eqref{BOmegadefn} are given by
\begin{align*}
B_{C,j} &=
\begin{cases}
0, 
&\hbox{if $j< a_1$,}
\\
1-qY_{a_1}Y^{-1}_{a_m}, &\hbox{if $j=a_1$,}
\\
\displaystyle{
\frac{1-t}{1-qt^{n-j+1}} (1-qY_{a_1}Y^{-1}_{a_m}), }
&\hbox{if $a_1< j < a_2$,}
\\
\displaystyle{
-q\frac{1-t}{1-qt^{n-j+1}}Y_{a_1}Y^{-1}_{a_m}
+Y_{a_1}Y^{-1}_{a_p} - t \frac{1-qt^{n-j}}{1-qt^{n-j+1}} Y_{a_1}Y^{-1}_{a_{p-1}}, }
&\hbox{if $j=a_p$ and $p\ne 1$,}
\\
\displaystyle{
-q\frac{1-t}{1-qt^{n-j+1}}Y_{a_1}Y^{-1}_{a_m}
+ \frac{1-t}{1-qt^{n-j+1}} Y_{a_1}Y^{-1}_{a_{p-1}}, }
&\hbox{if $a_{p-1} < j< a_p$ with $2<p$},
\\
\displaystyle{
\frac{(1-t)(1-q)}{1-qt^{n-j+1}}Y_{a_1}Y^{-1}_{a_m}, }
&\hbox{if $j>a_m$.}
\end{cases}
\end{align*}
Similar expressions can be given for the $\Omega_{C,J}$.
\qed
\end{remark}

\begin{example} \label{proofexample}
\textbf{A term in $X_j$ for $n=11$ and $j=7$.}
This is an example of the computation for the proof of part (a) of Theorem \ref{Monkops}.
In the expansion of 
\begin{align*}
X_7 &= T_6T_5T_4T_3T_2T_1 \tau^\vee_\pi T^{-1}_{10}T^{-1}_9T^{-1}_8 T^{-1}_7 \\
&= (\tau^\vee_6 - f^+_{76})\cdots (\tau^\vee_1-f^+_{21}) \tau^\vee_\pi
(\tau^\vee_{10}-f^-_{11,10})\cdots (\tau^\vee_7-f^-_{87}),
\end{align*}
the term coming from choosing the $-f^\pm_{i,i-1}$ from the 2nd, 5th, 8th and 9th factors is
\begin{align*}
\tau^\vee_6 (-f^+_{6,5}) &\tau^\vee_4  \tau^\vee_3 (-f^+_{3,2}) \tau^\vee_1
\tau^\vee_\pi
\tau^\vee_{10} (-f^-_{10,9}) (-f^-_{9,8}) \tau^\vee_7
\\
&= (-1)^4 \tau^\vee_6\tau^\vee_4\tau^\vee_3\tau^\vee_1 f^+_{6,3} f^+_{3,1}  
\tau^\vee_\pi \tau^\vee_{10} \tau^\vee_7  f^-_{10,9} f^-_{9,7} 
\\
&= (-1)^4 \tau^\vee_6\tau^\vee_4\tau^\vee_3\tau^\vee_1 
\tau^\vee_\pi f^+_{5,2} f^+_{2,11+K} \tau^\vee_{10} \tau^\vee_7  f^-_{10,9} f^-_{9,7} 
\\
&= (-1)^4 \tau^\vee_6\tau^\vee_4\tau^\vee_3\tau^\vee_1 
\tau^\vee_\pi \tau^\vee_{10} \tau^\vee_7  f^+_{5,2} f^+_{2,10+K} f^-_{10,9} f^-_{9,7}
\\
&= (-1)^4 \tau^\vee_{C,7} f^+_{5,2} f^+_{2,10+K} f^-_{10,9} f^-_{9,7},
\end{align*}
where
$C = \{2,5,7,9,10\}$ and $C^c=\{1,3,4,6,8,11\}$ so that
$$\tau^\vee_{C,7} 
= \tau^\vee_6\tau^\vee_4\tau^\vee_3\tau^\vee_1 
\tau^\vee_\pi \tau^\vee_{11-1} \tau^\vee_{8-1}
= \tau^\vee_6\tau^\vee_4\tau^\vee_3\tau^\vee_1 
\tau^\vee_\pi \tau^\vee_{10} \tau^\vee_7.
$$
Using
$$f^-_{52} 
= \frac{ t^{-\frac12}(1-t)Y_5Y^{-1}_2 }{ 1-Y_5Y^{-1}_2 }
= -\frac{ t^{-\frac12}(1-t) }{ 1-Y_2Y^{-1}_5 }
= -f^+_{25} 
$$
and
$$
f^+_{52} = \frac{t^{-\frac12}(1-t)}{1-Y_5Y^{-1}_2} 
= \frac{t^{-\frac12} Y_2Y^{-1}_5 (1-t)}{ Y_2Y^{-1}_5-1} 
= -\frac{t^{-\frac12} Y_2Y^{-1}_5 (1-t)}{ 1-Y_2Y^{-1}_5} 
= -  Y_2Y^{-1}_5 f^+_{25}
$$
gives
\begin{align*}
(-1)^4 f^+_{5,2} &f^+_{2,10+K} f^-_{10,9} f^-_{9,7}
= (-1) Y_2Y^{-1}_5 f^+_{25} f^+_{2,10+K} f^+_{9,10}f^+_{79}
= (-1) Y_2Y^{-1}_5 \frac{1}{f^+_{57}}  (f^+_{2,10+K}f^+_{25} f^+_{57} f^+_{79} f^+_{9,10})
\\
&= (-1)Y_2Y^{-1}_5
\frac{t^{-\frac12}(1-t)} {f^+_{57}} f_C = (Y_2Y^{-1}_7-Y_2Y^{-1}_5)f_C(Y)
= F_{C,7}(Y)f_C(Y).
\end{align*}
\qed
\end{example}

\begin{example} \label{invproofexample}
\textbf{An example of a term in $X^{-1}_j$ for $n=11$ and $j=7$.}
This is an example of the computation for the proof of part (d) of Theorem \ref{Monkops}.
In the expansion of 
\begin{align*}
X^{-1}_7 
&= T_7T_8T_9T_{10}(\tau^\vee_\pi)^{-1} T^{-1}_1T^{-1}_2T^{-1}_3T^{-1}_4T^{-1}_5T^{-1}_6 
\\
&= (\tau^\vee_7 - f^+_{87})\cdots (\tau^\vee_{10}-f^+_{11,10}) (\tau^\vee_\pi)^{-1}
(\tau^\vee_1-f^-_{21})\cdots (\tau^\vee_6-f^-_{76})
\end{align*}
the term coming from choosing the $-f^\pm_{i,i-1}$ from the 2nd, 5th, 8th and 9th factors is
\begin{align*}
\tau^\vee_7 (-f^+_{9,8}) &\tau^\vee_9  \tau^\vee_{10} 
(\tau^\vee_\pi)^{-1}
(-f^-_{2,1}) \tau^\vee_2\tau^\vee_3 (-f^-_{5,4})(-f^-_{6,5}) \tau^\vee_6
\\
&= (-1)^4 \tau^\vee_7\tau^\vee_9\tau^\vee_{10} f^+_{11,8}
(\tau^\vee_\pi)^{-1} 
\tau^\vee_2\tau^\vee_3 \tau^\vee_6 f^-_{4,1} f^-_{5,4} f^-_{7,5} 
\\
&= (-1)^4 \tau^\vee_7\tau^\vee_9\tau^\vee_{10}
(\tau^\vee_\pi)^{-1} f^+_{1-K,9}
\tau^\vee_2\tau^\vee_3 \tau^\vee_6 f^-_{4,1} f^-_{5,4} f^-_{7,5} 
\\
&= (-1)^4 \tau^\vee_7\tau^\vee_9\tau^\vee_{10}
(\tau^\vee_\pi)^{-1} 
\tau^\vee_2\tau^\vee_3 \tau^\vee_6 f^+_{1-K,9}f^-_{4,1} f^-_{5,4} f^-_{7,5} 
\\
&= (-1)^4 \rho_{D,7} f^+_{1-K,9} f^-_{7,5} f^-_{5,4} f^-_{4,1} 
\end{align*}
where
$D = \{1,4,5,7,9\}$ and $D^c = \{2,3,6,8,10,11\}$ so that
$$
\rho_{D,7}
= \tau^\vee_7\tau^\vee_9\tau^\vee_{10} (\tau^\vee_\pi)^{-1} 
\tau^\vee_2\tau^\vee_3 \tau^\vee_6
= \tau^\vee_{8-1}\tau^\vee_{10-1}\tau^\vee_{11-1} (\tau^\vee_\pi)^{-1} 
\tau^\vee_2\tau^\vee_3 \tau^\vee_6
$$
Using $f^-_{ij} = -f^+_{ji}$ and $f^+_{ji} = -Y_iY^{-1}_j f^+_{ij}$ gives
\begin{align*}
(-1)^4 f^+_{1-K,9} f^-_{7,5} f^-_{5,4} f^-_{4,1}
&= (-1) f^+_{1-K,9} f^+_{5,7} f^+_{4,5}f^+_{1,4}
= (-1) \frac{1}{f^+_{7,9}} f^+_{1-K,9}  f^+_{1,4}f^+_{4,5}f^+_{5,7}f^+_{7,9}
\\
&=(Y_7Y^{-1}_9-1) \frac{1}{t^{-\frac12}(1-t)} f^+_{1-K,9}  f^+_{1,4}f^+_{4,5}f^+_{5,7}f^+_{7,9}
= \Phi_{D,7}(Y)f_D(Y).
\end{align*}
\qed
\end{example}

\section{Monk rules for Macdonald polynomials}

Let $k^\uparrow\colon \ZZ^n\to \ZZ^n$ be the function which increments the $k$th coordinate by 1,
and let $k^\downarrow \colon \ZZ^n\to \ZZ^n$ be the function which decreases the $k$th coordinate by 1,
so that if $\mu = (\mu_1, \ldots, \mu_n)$ then
\begin{align*}
k^{\uparrow}\mu = k^{\uparrow}(\mu_1, \ldots, \mu_n) &= (\mu_1, \ldots, \mu_{k-1}, \mu_k+1, \mu_{k+1}, \ldots, \mu_n) \quad\hbox{and}
\\
k^{\downarrow}\mu = k^\downarrow(\mu_1, \ldots, \mu_n) &= (\mu_1, \ldots, \mu_{k-1}, \mu_k-1, \mu_{k+1}, \ldots, \mu_n).
\end{align*}
Let $j\in \{1, \ldots, n\}$ and let $C\subseteq \{1, \ldots, n\}$.
Write $C = \{a_1, \ldots, a_m\}$ with $a_1< a_2< \cdots < a_m$.
For $\mu = (\mu_1, \ldots, \mu_n) \in \ZZ^n$ define
\begin{equation}
\mathrm{rot}_C(\mu) = \gamma_C a_m^\uparrow \mu,
\qquad\hbox{where, in cycle notation, $\gamma_C=(a_1,\ldots, a_m)\in S_n$.}
\label{rotmuCdefn}
\end{equation}
Thus, $\mathrm{rot}_C(\mu)$ is the same as $\mu$ except that in $\mathrm{rot}_C(\mu)$ 
the parts of $\mu$ indexed by the elements of $C$ have been rotated and $1$ has been added to 
$\mu_{a_m}$.
\begin{equation}
\mu = 
\begin{matrix}\begin{tikzpicture}[scale=.9]
\draw (0,0) circle (2cm);
\foreach \x in {1, ..., 24} {\draw (0,0) to (\x*360/24:2.1cm);}
\draw[thick, -latex] (0,0) to (0,2.1);
\foreach \x/\y in {0/1,1/2,2/3,23/n,22/{n\sm1}}{\node at (90-\x*360/24:2.5cm) {\scriptsize$\mu_{\y}$};}
\foreach \x in {3, 3.2, 2.8, 21, 21.2, 20.8}{ \node[D] at (90-\x*360/24:2.5cm) {};}
\foreach \x in {5,7,13,14,17,20} {\draw[very thick,red!80!black] (0,0) to (90 - \x*360/24:2.1cm);}
\foreach \x/\y in {5/{a_1},7/{a_2},20/{a_m}}{\node[red!80!black] at (90-\x*360/24:2.5cm) {\scriptsize$\mu_{\y}$};}
\node[white] at (90-13*360/24:2.5cm) {\scriptsize$\phantom{\mu_{a_2}}$};
\end{tikzpicture}\end{matrix}
\qquad 
\mathrm{rot}_C(\mu) = 
\begin{matrix}\begin{tikzpicture}[scale=.9]
\draw (0,0) circle (2cm);
\foreach \x in {1, ..., 24} {\draw (0,0) to (\x*360/24:2.1cm);}
\draw[thick, -latex] (0,0) to (0,2.1);
\foreach \x/\y in {0/1,1/2,2/3,23/n,22/{n\sm1}}{\node at (90-\x*360/24:2.5cm) {\scriptsize$\mu_{\y}$};}
\foreach \x in {3, 3.2, 2.8, 21, 21.2, 20.8}{ \node[D] at (90-\x*360/24:2.5cm) {};}
\foreach \x in {5,7,13,14,17,20} {\draw[very thick,red!80!black] (0,0) to (90 - \x*360/24:2.1cm);}
\foreach \x/\y in {7/{a_1},13/{a_2},20/{a_{m\sm1}}}{\node[red!80!black] at (90-\x*360/24:2.5cm) {\scriptsize$\mu_{\y}$};}
\node[red!80!black] at (90-5*360/24:2.5cm) {\scriptsize\hspace{15pt} $\mu_{a_m} + 1$};
\end{tikzpicture}\end{matrix}
\label{rotmuCpic}
\end{equation}
Let $\mathrm{rrot}_C$ be the inverse operation to $\mathrm{rot}_C$ so that
$\mathrm{rrot}_C(\mathrm{rot_C}(\mu))=\mu$ and 
$\mathrm{rrot}_C(\mu)$ is the same as $\mu$ except that in $\mathrm{rrot}_C(\mu)$ 
the parts of $\mu$ indexed by the elements of $C$ have been rotated counterclockwise
and $1$ has been \emph{subtracted} from
$\mu_{a_1}$.

For $k\in \{1, \ldots, n\}$ such that $k\not\in C$ define
\begin{equation}
b(k) = \begin{cases}
\mu_{a_m}+1, &\hbox{if $1\le k<a_1$,} \\
\mu_{a_i}, &\hbox{if $a_i< k < a_{i+1}$,}
\\
\mu_{a_m}, &\hbox{if $a_m< k\le n$,}
\end{cases}
\quad\hbox{and}\quad
c(k) = \begin{cases}
v_{a_m^\uparrow \mu}(a_m), &\hbox{if $1\le k<a_1$,} \\
v_\mu(a_i), &\hbox{if $a_i< k < a_{i+1}$,}
\\
v_\mu(a_m), &\hbox{if $a_m< k\le n$,}
\end{cases}
\label{bkckdefns}
\end{equation}
and
$$d(k) = \begin{cases}
\mu_{a_1}-1, &\hbox{if $a_m< k\le n$,} \\
\mu_{a_i}, &\hbox{if $a_{i-1}< k < a_i$,}
\\
\mu_{a_1}, &\hbox{if $1\le k < a_1$,}
\end{cases}
\quad\hbox{and}\quad
e(k) = \begin{cases}
v_{a_1^\downarrow \mu}(a_1), &\hbox{if $a_m<k\le n$,} \\
v_\mu(a_i), &\hbox{if $a_{i-1}< k < a_i$,}
\\
v_\mu(a_1), &\hbox{if $1\le k< a_1$,}
\end{cases}
$$
Keeping $k\in \{1, \ldots, n\}$ such that $k\not\in C$ define
$$\wt_\mu(C, k) = \begin{cases}
0, &\hbox{if $b(k) =\mu_k$,}
\\
1, &\hbox{if $b(k)>\mu_k$,}
\\
\displaystyle{
t \frac{(1-q^{\mu_k-b(k)} t^{v_\mu(k)-c(k)+1})(1-q^{\mu_k-b(k)}t^{v_\mu(k)-c(k)-1})}
{(1-q^{\mu_k-b(k)}t^{v_\mu(k)-c(k)})^2},
}
&\hbox{if $b(k)<\mu_k$,}
\end{cases}
$$
and
$$\mathrm{rwt}_\mu(C, k) = \begin{cases}
0, &\hbox{if $d(k) =\mu_k$,}
\\
t^{-1}, &\hbox{if $d(k)>\mu_k$,}
\\
\displaystyle{
\frac{(1-q^{\mu_k-d(k)} t^{v_\mu(k)-e(k)+1})(1-q^{\mu_k-d(k)}t^{v_\mu(k)-e(k)-1})}
{(1-q^{\mu_k-d(k)}t^{v_\mu(k)-e(k)})^2},
}
&\hbox{if $d(k)<\mu_k$.}
\end{cases}
$$
For $k\in \{1, \ldots, n\}$ such that $k\in C$ define
\begin{equation}
\wt_\mu(C,k) = \mathrm{rwt}_\mu(C,k) = \begin{cases}
\displaystyle{
\frac{1-t}{1-q^{\mu_{a_{i+1}}-\mu_{a_i}}t^{v_\mu(a_{i+1})-v_\mu(a_i)}},
}
&\hbox{if $k=a_i$ and $i\ne m$,}
\\
\displaystyle{
\frac{1}{1-q^{\mu_{a_m}-\mu_{a_1}+1}t^{v_\mu(a_m)-v_\mu(a_1)}},
}
&\hbox{if $k = a_m$.}
\end{cases}
\label{wtmuCkdefn}
\end{equation}
Then define
\begin{equation}
\wt_\mu(C) = t^{-\#\{ i\ |\ \mu_i> \mu_{a_m}\}} 
\prod_{i=1}^n \wt_\mu(C,k),
\label{wtmuCdefn}
\end{equation}
and
\begin{equation}
\mathrm{rwt}_\mu(C) = t^{\#\{ i\ |\ \mu_i> \mu_{a_1}\}} 
\prod_{i=1}^n \mathrm{rwt}_\mu(C,k),
\label{rwtmuCdefn}
\end{equation}

\begin{thm}  \label{Monkthm} \textbf{(Monk rules for Macdonald polynomials)}
Let $j\in \{1,\ldots, n\}$ and $\mu\in \ZZ_{\ge 0}^n$.  Let $E_\mu$ denote the
electronic Macdonald polynomial indexed by $\mu$ in $\CC[x_1^{\pm1},\ldots, x_n^{\pm1}]$.
Let
\begin{align*}
F_\mu(C,j) 
&=
\begin{cases}
0, &\hbox{if $j\not\in C$.}
\\
1-q^{\mu_{a_m}-\mu_{a_1}+1 }t^{v_\mu(a_m)-v_\mu(a_1)}, 
&\hbox{if $j=a_p$ and $p=1$,}
\\
q^{\mu_{a_p}-\mu_{a_1 } }t^{v_\mu(a_p)-v_\mu(a_1) } 
-q^{\mu_{a_{p-1}}-\mu_{a_1 } }t^{v_\mu(a_{p-1})-v_\mu(a_1)},
&\hbox{ if $j=a_p$ and $p\ne1$,}
\end{cases}
\\
\\
A_\mu(C,j) 
&= \begin{cases}
0, 
&\hbox{if $j<a_1$,}
\\
q^{\mu_{a_p}-\mu_{a_1 } }t^{v_\mu(a_p)-v_\mu(a_1) } 
-q^{\mu_{a_m}-\mu_{a_1}+1 }t^{v_\mu(a_m)-v_\mu(a_1)},
&\hbox{if $a_p\le j< a_{p+1}$,}
\\
(1-q)q^{\mu_{a_m}-\mu_{a_1 } }t^{v_\mu(a_m)-v_\mu(a_1) },
&\hbox{if $j>a_m$,}
\end{cases}
\end{align*}
and
\begin{equation}
B_\mu(C,j) = F_\mu(C,j)+ \frac{1-t}{1-qt^{n-j+1}} A_\mu(C,j).
\label{FABdefn}
\end{equation}
Let
\begin{align*}
\Phi_\mu(C,j) 
&=
\begin{cases}
0, &\hbox{if $j\not\in C$.}
\\
1-q^{\mu_{a_m}-\mu_{a_1}+1 }t^{v_\mu(a_m)-v_\mu(a_1)}, 
&\hbox{if $j=a_p$ and $p=m$,}
\\
q^{\mu_{a_p}-\mu_{a_1}}t^{v_\mu(a_p)-v_\mu(a_1)} 
-q^{\mu_{a_{p-1}}-\mu_{a_1 } }t^{v_\mu(a_{p-1})-v_\mu(a_1)},
&\hbox{ if $j=a_p$ and $p\ne m$,}
\end{cases}
\\
\\
\Psi_\mu(C,j) 
&= \begin{cases}
0, 
&\hbox{if $j<a_1$,}
\\
q^{\mu_{a_m}-\mu_{a_p } }t^{v_\mu(a_m)-v_\mu(a_p) } 
-q^{\mu_{a_m}-\mu_{a_1}+1 }t^{v_\mu(a_m)-v_\mu(a_1)},
&\hbox{if $a_{p-1}< j\le a_p$,}
\\
(1-q)q^{\mu_{a_m}-\mu_{a_1 } }t^{v_\mu(a_m)-v_\mu(a_1) },
&\hbox{if $j\le a_1$,}
\end{cases}
\end{align*}
and
\begin{equation}
\Omega_\mu(C,j) = \Phi_\mu(C,j)+ \frac{1-t}{1-qt^j} \Psi_\mu(C,j).
\label{FABinvdefn}
\end{equation}
Let $\mathrm{rot}_\mu(C)$ and $\wt_\mu(C)$ as  in \eqref{rotmuCdefn} 
and \eqref{wtmuCdefn},
and let $\mathrm{rrot}_\mu(C)$ and $\mathrm{rwt}_\mu(C)$ be as defined in \eqref{rotmuCdefn} 
and \eqref{rwtmuCdefn}. 
Then
\item[(a)] 
$$x_j E_\mu = \sum_{C\subseteq \{1, \ldots, n\} \atop C\cap\{j\}\ne\emptyset} 
F_\mu(C,j) \wt_\mu(C) E_{\mathrm{rot}_C(\mu)},
$$
\item[(b)] 
$$(x_1+\cdots+x_j) E_\mu = \sum_{C\subseteq \{1, \ldots, n\} \atop C\cap \{1,\ldots, j\}\ne \emptyset} 
A_\mu(C,j) \wt_\mu(C) E_{\mathrm{rot}_C(\mu)},
$$
\item[(c)] 
$$E_{\varepsilon_j} E_\mu = \sum_{C\subseteq \{1, \ldots, n\} \atop C\cap \{1,\ldots, j\}\ne \emptyset} 
B_\mu(C,j) \wt_\mu(C) E_{\mathrm{rot}_C(\mu)},
$$
\item[(d)] 
$$x^{-1}_j E_\mu = \sum_{C\subseteq \{1, \ldots, n\} \atop C\cap\{j\}\ne\emptyset} 
\Phi_\mu(C,j) \mathrm{rwt}_\mu(C) E_{\mathrm{rrot}_C(\mu)},
$$
\item[(e)] 
$$(x^{-1}_j+\cdots+x^{-1}_n) E_\mu 
= \sum_{C\subseteq \{1, \ldots, n\} \atop C\cap \{j,\ldots, n\}\ne \emptyset} 
\Psi_\mu(C,j) \mathrm{rwt}_\mu(C) E_{\mathrm{rrot}_C(\mu)},
$$
\item[(f)] 
$$E_{-\varepsilon_j} E_\mu = \sum_{C\subseteq \{1, \ldots, n\} \atop C\cap \{j,\ldots, n\}\ne \emptyset} 
\Omega_\mu(C,j) \mathrm{rwt}_\mu(C) E_{\mathrm{rrot}_C(\mu)},
$$
\end{thm}
\begin{proof}
From \cite[(4.1) and (4.2)]{GR21}, if $\mu_i>\mu_{i+1}$ then
\begin{equation}
\begin{array}{l}
t^{\frac12}\tau^\vee_i E_\mu = E_{s_i\mu}
\quad\hbox{and} \\
\displaystyle{
t^{\frac12}\tau^\vee_i E_{s_i\mu} 
= t \frac{
(1-q^{\mu_i-\mu_{i+1}}t^{v_\mu(i)-v_\mu(i+1)+1}) 
(1-q^{\mu_i-\mu_{i+1} } t^{v_\mu(i)-v_\mu(i+1)-1} ) }
{(1-q^{\mu_i-\mu_{i+1}}t^{v_\mu(i)-v_\mu(i+1)})^2}
E_\mu.
}
\end{array}
\label{tauiaction}
\end{equation}
From \cite[(3.5)]{GR21},
\begin{equation}
\tau^\vee_\pi E_\mu = t^{\frac12(n-1)-\#\{i\in \{1,\ldots, n-1\}\ |\ \mu_i\le \mu_n\}} E_{\pi\mu}.
\label{taupiaction}
\end{equation}
If the complement of $C$ in $\{1, \ldots n\}$ is
$$C^c = \{b_1, \ldots, b_{n-m}\}\quad\hbox{with}\quad
b_1<\cdots < b_r < j< b_{r+1}< \cdots < b_{n-m}
$$
then
$
\tau^\vee_{C,j} 
= \tau^\vee_{b_r}\tau^\vee_{b_{r-1}}\cdots \tau^\vee_{b_1}
\tau^\vee_\pi \tau^\vee_{b_{n-1}-1}\cdots \tau^\vee_{b_{r+1}-1}
$
and using \eqref{tauiaction} and \eqref{taupiaction} gives
\begin{align}
t^{\frac12(n-m)}\tau^\vee_{C,j} E_\mu
&= t^{\frac12}\tau^\vee_{b_r}t^{\frac12}\tau^\vee_{b_{r-1}}\cdots t^{\frac12}\tau^\vee_{b_1}
\tau^\vee_\pi t^{\frac12}\tau^\vee_{b_{n-m}-1}\cdots t^{\frac12}\tau^\vee_{b_{r+1}-1} E_\mu
\nonumber \\
&= t^{\frac12(n-1)-\#\{\mu_i>\mu_{a_m}\}} \Big(\prod_{k\not\in C} \wt_\mu(C,k)\Big)
E_{\mathrm{rot}_C(\mu)}.
\label{tauCjonEmu}
\end{align}
An example of the step-by-step computation of $\tau^\vee_{C,j} E_\mu$ is given in Example \ref{proofexample2}.

Let $f_C(Y)$ and $F_{C,j}(Y)$ be as defined in \eqref{fCdefn} and \eqref{FCjYdefn}, 
and let $\ev^t_\mu$ be the evaluation map defined in \eqref{evtmudefn}.
Since
$$\ev^t_\mu(Y_iY^{-1}_j)
= q^{\mu_j-\mu_i}t^{v_\mu(j)-v_\mu(i)}
\qquad\hbox{and}\qquad
\ev^t_\mu\Big(\frac{1-t}{1-Y_iY^{-1}_j} \Big)
=\frac{1-t}{1- q^{\mu_j-\mu_i}t^{v_\mu(j)-v_\mu(i)}}
$$
then comparing \eqref{wtmuCkdefn} and \eqref{fCdefn} gives
\begin{equation}
\ev^t_\mu(f_C(Y)) = t^{-\frac12(m-1)}\prod_{k\in C} \wt_\mu(C,k)
\qquad\hbox{and}\qquad
\ev^t_\mu(F_{C,j}(Y)) = F_\mu(C,j).
\label{evals}
\end{equation}
Using \eqref{evequalsev} on the expression in Theorem \ref{Monkops}(a) and inserting
\eqref{evals} and \eqref{tauCjonEmu} gives
\begin{align*}
&x_j E_\mu 
= X_jE_\mu = \sum_{C\subseteq \{1, \ldots, n\}\atop j\in C}
\tau^\vee_{C,j} F_{C.j}(Y)f_C(Y) E_\mu 
\\
&= \sum_{C\subseteq \{1, \ldots, n\}\atop j\in C}
\tau^\vee_{C,j} \ev^t_\mu(F_{C,j}(Y))\ev^t_\mu(f_C(Y)) E_\mu 
\\
&= \sum_{C\subseteq \{1, \ldots, n\}\atop j\in C}
F_\mu(C,j)t^{-\frac12(m-1)}\Big(\prod_{k\in C} \wt_\mu(C,k)\Big) \tau^\vee_{C,j} E_\mu
\\
&= \sum_{C\subseteq \{1, \ldots, n\}\atop j\in C}
F_\mu(C,j)t^{-\frac12(m-1)}\Big(\prod_{k\in C} \wt_\mu(C,k)\Big)
t^{\frac12(n-1)-\#\{\mu_i<\mu_{a_m}\}} \Big(\prod_{k\not\in C} \wt_\mu(C,k)\Big)
t^{-\frac12(n-m)}E_{\mathrm{rot}_C(\mu)}
\\
&= \sum_{C\subseteq \{1, \ldots, n\}\atop j\in C}
F_\mu(C,j)\wt_\mu(C) E_{\mathrm{rot}_C(\mu)}.
\end{align*}
This completes the proof of (a).  The proof of the remaining parts is similar, using
parts (b)-(f) of Theorem \ref{Monkops}.
\end{proof}

\begin{example} \label{proofexample2}
\textbf{An example of the computation of $\tau^\vee_{C,j}E_\mu$.}  Let $n=11$ and $j=7$ and
$$C = \{2,5,7,9,10\},
\qquad\hbox{so that}\quad
\tau^\vee_{C,7} 
= \tau^\vee_6\tau^\vee_4\tau^\vee_3\tau^\vee_1 
\tau^\vee_\pi \tau^\vee_{10} \tau^\vee_7
= \tau^\vee_6\tau^\vee_4\tau^\vee_3\tau^\vee_1 
\tau^\vee_\pi \tau^\vee_{11-1} \tau^\vee_{8-1}$$
and $C^c=\{1,3,4,6,8,11\}$.
Then, using \eqref{tauiaction} and \eqref{taupiaction},
\begin{align*}
t^{\frac62}&\tau^\vee_{C,7} E_\mu
= t^{\frac12}\tau^\vee_6t^{\frac12}\tau^\vee_4t^{\frac12}\tau^\vee_3t^{\frac12}\tau^\vee_1 
\tau^\vee_\pi t^{\frac12}\tau^\vee_{11-1} t^{\frac12}\tau^\vee_{8-1}
E_{(\mu_1, {\color{red}\mu_2}, \mu_3, \mu_4, {\color{red}\mu_5}, \mu_6, {\color{red}\mu_7}, \mu_8,
{\color{red}\mu_9}, {\color{red}\mu_{10}}, \mu_{11})}
\\
&= \wt_\mu(C,8) t^{\frac12}\tau^\vee_6 t^{\frac12}\tau^\vee_4t^{\frac12}\tau^\vee_3t^{\frac12}\tau^\vee_1 
\tau^\vee_\pi t^{\frac12}\tau^\vee_{11-1} 
E_{(\mu_1, {\color{red}\mu_2}, \mu_3, \mu_4, {\color{red}\mu_5}, \mu_6, \mu_8, {\color{red}\mu_7}, 
{\color{red}\mu_9}, {\color{red}\mu_{10}}, \mu_{11})}
\\
&= \wt_\mu(C,8) \wt_\mu(C,11) t^{\frac12}\tau^\vee_6t^{\frac12}\tau^\vee_4t^{\frac12}\tau^\vee_3t^{\frac12}\tau^\vee_1 
\tau^\vee_\pi 
E_{(\mu_1, {\color{red}\mu_2}, \mu_3, \mu_4, {\color{red}\mu_5}, \mu_6, \mu_8, {\color{red}\mu_7}, 
{\color{red}\mu_9}, \mu_{11}, {\color{red}\mu_{10}})}
\\
&= t^{\frac12(11-1)-\#\{\mu_i<\mu_{10}\}}
\wt_\mu(C,8) \wt_\mu(C,11) t^{\frac12}\tau^\vee_6t^{\frac12}\tau^\vee_4t^{\frac12}\tau^\vee_3t^{\frac12}\tau^\vee_1 
E_{({\color{red}\mu_{10}+1}, \mu_1, {\color{red}\mu_2}, \mu_3, \mu_4, {\color{red}\mu_5}, \mu_6, \mu_8, {\color{red}\mu_7}, {\color{red}\mu_9}, \mu_{11})}
\\
&= t^{5-\#\{\mu_i<\mu_{10}\}} \big(\prod_{k\in \{1,8,11\}} \wt_\mu(C,k)\Big)
t^{\frac12}\tau^\vee_6t^{\frac12}\tau^\vee_4t^{\frac12}\tau^\vee_3
E_{(\mu_1, {\color{red}\mu_{10}+1}, {\color{red}\mu_2}, \mu_3, \mu_4, {\color{red}\mu_5}, \mu_6, \mu_8, {\color{red}\mu_7}, {\color{red}\mu_9}, \mu_{11})}
\\
&= t^{5-\#\{\mu_i<\mu_{10}\}} \big(\prod_{k\in \{1,3, 8,11\}} \wt_\mu(C,k)\Big)
t^{\frac12}\tau^\vee_6t^{\frac12}\tau^\vee_4
E_{(\mu_1, {\color{red}\mu_{10}+1}, \mu_3, {\color{red}\mu_2}, \mu_4, {\color{red}\mu_5}, \mu_6, \mu_8, {\color{red}\mu_7}, {\color{red}\mu_9}, \mu_{11})}
\\
&= t^{5-\#\{\mu_i<\mu_{10}\}} \big(\prod_{k\in \{1,3, 4, 8,11\}} \wt_\mu(C,k)\Big)
t^{\frac12}\tau^\vee_6
E_{(\mu_1, {\color{red}\mu_{10}+1}, \mu_3, \mu_4, {\color{red}\mu_2}, {\color{red}\mu_5}, \mu_6, \mu_8, {\color{red}\mu_7}, {\color{red}\mu_9}, \mu_{11})}
\\
&= t^{5-\#\{\mu_i<\mu_{10}\}} \big(\prod_{k\in \{1,3, 4, 6, 8,11\}} \wt_\mu(C,k)\Big)
E_{(\mu_1, {\color{red}\mu_{10}+1}, \mu_3, \mu_4, {\color{red}\mu_2}, \mu_6, {\color{red}\mu_5}, \mu_8, {\color{red}\mu_7}, {\color{red}\mu_9}, \mu_{11})}
\\
&= t^{5-\#\{\mu_i<\mu_{10}\}} \big(\prod_{k\in C^c} \wt_\mu(C,k)\Big)
E_{\mathrm{rot}_C(\mu)}.
\end{align*} 
The red entries correspond to the parts specified by $C$ which are rotated to get
$\mathrm{rot}_\mu(C)$ as in the picture in 
\eqref{rotmuCpic}.
\qed
\end{example}

\begin{example} \label{invexample2}
\textbf{An example of the computation of $\rho_{D,j}E_\mu$.}  Let $n=11$ and $j=7$ and
$$D = \{1,4,5,7,9\},
\qquad\hbox{so that}\quad
\rho_{D,7} 
= \tau^\vee_7\tau^\vee_9\tau^\vee_{10}
(\tau^\vee_\pi)^{-1} \tau^\vee_2 \tau^\vee_3 \tau^\vee_6$$
and $D^c=\{2,3,6,8,10,11\}$.
Then, using \eqref{tauiaction} and \eqref{taupiaction},
\begin{align*}
t^{\frac62}&\rho_{D,7} E_\mu
= t^{\frac12}\tau^\vee_7t^{\frac12}\tau^\vee_9t^{\frac12}\tau^\vee_{10}
(\tau^\vee_\pi)^{-1} t^{\frac12}\tau^\vee_2 t^{\frac12}\tau^\vee_3 t^{\frac12}\tau^\vee_6
E_{({\color{red}\mu_1}, \mu_2,\mu_3, {\color{red}\mu_4}, {\color{red}\mu_5}, \mu_6, {\color{red}\mu_7}, \mu_8,
{\color{red}\mu_9}, \mu_{10}, \mu_{11})}
\\
&= \mathrm{rwt}_\mu(D,6)
t^{\frac12}\tau^\vee_7t^{\frac12}\tau^\vee_9t^{\frac12}\tau^\vee_{10}
(\tau^\vee_\pi)^{-1} t^{\frac12}\tau^\vee_2 t^{\frac12}\tau^\vee_3
E_{({\color{red}\mu_1}, \mu_2,\mu_3, {\color{red}\mu_4}, {\color{red}\mu_5},  {\color{red}\mu_7}, \mu_6, \mu_8,
{\color{red}\mu_9}, \mu_{10}, \mu_{11})}
\\
&= \mathrm{rwt}_\mu(D,3)\mathrm{rwt}_\mu(D,6)
t^{\frac12}\tau^\vee_7t^{\frac12}\tau^\vee_9t^{\frac12}\tau^\vee_{10}
(\tau^\vee_\pi)^{-1} t^{\frac12}\tau^\vee_2
E_{({\color{red}\mu_1}, \mu_2, {\color{red}\mu_4}, \mu_3, {\color{red}\mu_5},  {\color{red}\mu_7}, \mu_6, \mu_8,
{\color{red}\mu_9}, \mu_{10}, \mu_{11})}
\\
&= \Big(\prod_{k\in \{2,3,6\} } \mathrm{rwt}_\mu(D,k) \Big)
t^{\frac12}\tau^\vee_7t^{\frac12}\tau^\vee_9t^{\frac12}\tau^\vee_{10}
(\tau^\vee_\pi)^{-1}
E_{({\color{red}\mu_1},  {\color{red}\mu_4}, \mu_2, \mu_3, {\color{red}\mu_5},  {\color{red}\mu_7}, \mu_6, \mu_8,
{\color{red}\mu_9}, \mu_{10}, \mu_{11})}
\\
&= t^{-\frac12(11-1)+\#\{\mu_i<\mu_1\}} \Big(\prod_{k\in \{2,3,6\} } \mathrm{rwt}_\mu(D,k) \Big)
t^{\frac12}\tau^\vee_7t^{\frac12}\tau^\vee_9t^{\frac12}\tau^\vee_{10}
E_{(  {\color{red}\mu_4}, \mu_2, \mu_3, {\color{red}\mu_5},  {\color{red}\mu_7}, \mu_6, \mu_8,
{\color{red}\mu_9}, \mu_{10}, \mu_{11}, {\color{red}\mu_1-1})}
\\
&= t^{-5+\#\{\mu_i<\mu_1\}} \Big(\prod_{k\in \{2,3,6,11\} } \mathrm{rwt}_\mu(D,k) \Big)
t^{\frac12}\tau^\vee_7t^{\frac12}\tau^\vee_9
E_{(  {\color{red}\mu_4}, \mu_2, \mu_3, {\color{red}\mu_5},  {\color{red}\mu_7}, \mu_6, \mu_8,
{\color{red}\mu_9}, \mu_{10},  {\color{red}\mu_1-1}, \mu_{11})}
\\
&= t^{-5+\#\{\mu_i<\mu_1\}}\Big(\prod_{k\in \{2,3,6,10,11\} } \mathrm{rwt}_\mu(D,k) \Big)
t^{\frac12}\tau^\vee_7
E_{(  {\color{red}\mu_4}, \mu_2, \mu_3, {\color{red}\mu_5},  {\color{red}\mu_7}, \mu_6, \mu_8,
{\color{red}\mu_9},   {\color{red}\mu_1-1}, \mu_{10}, \mu_{11})}
\\
&= t^{-5+\#\{\mu_i<\mu_1\}}\Big(\prod_{k\in \{2,3,6, 8, 10,11\} } \mathrm{rwt}_\mu(D,k) \Big)
E_{(  {\color{red}\mu_4}, \mu_2, \mu_3, {\color{red}\mu_5},  {\color{red}\mu_7}, \mu_6, 
{\color{red}\mu_9}, \mu_8,  {\color{red}\mu_1-1}, \mu_{10}, \mu_{11})}
\\
&= t^{-5+\#\{\mu_i<\mu_1\}} \big(\prod_{k\in D^c} \mathrm{rwt}_\mu(D,k)\Big)
E_{\mathrm{rrot}_D(\mu)}.
\end{align*}\qed
\end{example}

\section{Specializations of the Monk rule}

Specializations of the electronic Macdonald polynomials at $q=0$, $t=0$, $q=\infty$ and $t=\infty$
are of interest.  For example,
\begin{enumerate}
\item[(a)] $E_\mu(0,t)$ are the Iwahori-spherical functions of \cite{Ion04}, also called 
$t$-deformations of Demazure characters and Demazure atoms in \cite{Al16}; 
\item[(b)] $E_\mu(q,0)$ are (level 1 or level 0) affine Demazure characters,
or (affine) key polynomials, or non-symmetric $q$-Whittaker polynomials
(see \cite{Ion01}, \cite{MRY19} and \cite{AG20}).
\item[(c)] $E_\mu(0,0)$ are Demazure characters or (finite) key polynomials.
The finite key polynomials are special cases of the affine key polynomials.
\end{enumerate}
By appropriately packaging the weights in the Monk formulas in Theorem \ref{Monkthm} it is
easy to specialize these formulas and obtain formulas at $t=0$ and $q=0$.  
Proposition \ref{propIntExp}
does this repackaging for the product $x_jE_\mu$ and the resulting formulas at
$q=0$ and $t=0$ are given in Corollary \ref{spclzdMonk}.
Similar formulas could be given for the other products in Theorem \ref{Monkthm}
and also for specializations at $t=\infty$ and $q=\infty$ (by packaging the coefficients
in terms of $q^{-1}$ and $t^{-1}$ and then setting $q^{-1}=0$ and/or $t^{-1}=0$).

Let $\mu\in \ZZ^n$ and let $j\in \{1, \ldots, n\}$.  Let $C\subseteq \{1, \ldots, n\}$ and let
$$C = \{a_1, a_2, \ldots, a_m\}
\qquad\hbox{with}\quad 1\le a_1< a_2< \cdots <a_m \le n.$$
For parsing the following definitions it is useful to note that
\begin{align*}
&\mu_{a_{i-1}} > \mu_{a_i}
&&\hbox{if and only if}
&&v_\mu(a_{i-1})>v_\mu(a_i), &&\hbox{and} \\
&\mu_{a_1}<\mu_{a_m}+1
&&\hbox{if and only if}
&&v_\mu(a_1)<v_\mu(a_m).
\end{align*}
Assume $j\in C$ and let $p\in \{1, \ldots, m\}$ be given by $j=a_p$.  Let
\begin{align*}
S' &= \#\{ i\in \{2, \ldots, m\}\ |\ \mu_{a_{i-1}}>\mu_{a_i}\} + \begin{cases}
1, &\hbox{if $v_\mu(a_1)>v_\mu(a_m)$,} \\
0, &\hbox{if $v_\mu(a_1)<v_\mu(a_m)$,} 
\end{cases}
\\
A' &= \sum_{i\in \{2, \ldots, m\}\atop \mu_{a_{i-1}}>\mu_{a_i}} (\mu_{a_{i-1}}-\mu_{a_i})
+\begin{cases}
\mu_{a_1}-(\mu_{a_m}+1), &\hbox{ if $\mu_{a_1}\ge \mu_{a_m}+1$,} \\
0, &\hbox{if $\mu_{a_1}<\mu_{a_m}+1$},
\end{cases}
\\
B' &= -\{ i\ |\mu_i>\mu_{a_m}\}
+\{ k\not\in C\ |\ b(k)<\mu_k\}
\\
&\qquad
+\sum_{i\in \{2, \ldots, m\}\atop v_\mu(a_{i-1})>v_\mu(a_i)} (v_\mu(a_{i-1})-v_\mu(a_i))
+\begin{cases}
v_\mu(a_1)-v_\mu(a_m), &\hbox{if $v_\mu(a_1)>v_\mu(a_m)$,} \\
0, &\hbox{if $v_\mu(a_1)<v_\mu(a_m)$.} 
\end{cases}
\end{align*}
Then define
\begin{align}
S_{j,\mu}(C) &= S' + \begin{cases}
1, &\hbox{if $p\ne 1$ and $\mu_{a_p}\ge \mu_{a_{p-1}}$,} \\
0, &\hbox{if $p\ne 1$ and $\mu_{a_p}< \mu_{a_{p-1}}$,} \\
1, &\hbox{if $p=1$ and $\mu_{a_m}+1\le \mu_{a_1}$,} \\
0, &\hbox{if $p=1$ and $\mu_{a_m}+1>\mu_{a_1}$,} 
\end{cases}
\label{Sdefn}
\\
A_{j,\mu}(C) &= A' + \begin{cases}
\mu_{a_{p-1}}-\mu_{a_1}, &\hbox{if $p\ne 1$ and $\mu_{a_p}\ge \mu_{a_{p-1}}$,} \\
\mu_{a_p}-\mu_{a_1}, &\hbox{if $p\ne 1$ and $\mu_{a_p}< \mu_{a_{p-1}}$,} \\
-(\mu_{a_1}-(\mu_{a_m}+1)), &\hbox{if $p=1$ and $\mu_{a_m}+1\le \mu_{a_1}$,} \\
0, &\hbox{if $p=1$ and $\mu_{a_m}+1>\mu_{a_1}$,} 
\end{cases}
\label{Adefn}
\\
B_{j,\mu}(C) &= B' + \begin{cases}
v_\mu(a_{p-1})-v_\mu(a_1), &\hbox{if $p\ne 1$ and $v_\mu(a_{p-1}) < v_\mu(a_p)$,} \\
v_\mu(a_p)-v_\mu(a_1), &\hbox{if $p\ne 1$ and $v_\mu(a_{p-1})>v_\mu(a_p)$,} \\
-(v_\mu(a_1) - v_\mu(a_m)), &\hbox{if $p=1$ and $v_\mu(a_1)>v_\mu(a_p)$,} \\
0, &\hbox{if $p=1$ and $v_\mu(a_1)<v_\mu(a_p)$.}
\end{cases}
\label{Bdefn}
\end{align}

\begin{remark}  The statistics $S_{j,\mu}(C)$, 
$A_{j,\mu}(C)$ and $B_{j,\mu}(C)$ are
interesting statistics on 
$\mu$ and on the permutation $v_\mu$.
What properties do these statistics have?  How do they change when parts of $\mu$ are 
interchanged?
\end{remark}

\begin{prop}  \label{propIntExp}
Let $\mu\in \ZZ^n$ and let $j\in \{1, \ldots, n\}$.  Let $C\subseteq \{1, \ldots, n\}$ and let
$$C = \{a_1, a_2, \ldots, a_m\}
\qquad\hbox{with}\quad 1\le a_1< a_2< \cdots <a_m \le n.$$
Assume $j\in C$ and let $p\in \{1, \ldots, m\}$ be given by $j=a_p$.  
Let 
$$a_0 = a_m, \qquad
\gamma_0 = \mu_{a_m}+1,
\qquad\hbox{and}\qquad
\gamma_i = \mu_{a_i}, \quad\hbox{for $i\in \{1, \ldots, m\}$,}
$$
and define
$$
W_{\mu,k\in C} = \left(\prod_{i\in \{1, \ldots, m\}\atop i\ne p}
\frac{1-t}{1-q^{\vert \gamma_i - \gamma_{i-1}\vert }t^{\vert v_\mu(a_i)-v_\mu(a_{i-1})\vert }}
\right)
$$
For $k\not\in C$ let $b(k)$ and $c(k)$ be as defined in \eqref{bkckdefns} and let
$$
W_{\mu,k\not\in C}
=\left( \prod_{k\not\in C\atop \mu_k>b(k)} 
\frac{(1-q^{\mu_k-b(k)}t^{v_\mu(k)-c(k)+1})(1-q^{\mu_k-b(k)}t^{v_\mu(k)-c(k)-1})}
{(1-q^{\mu_k-b(k)}t^{v_\mu(k)-c(k)})^2}
\right).
$$
Let $S_{j,\mu}(C)$, $A_{j,\mu}(C)$ and $B_{j,\mu}(C)$ be as defined in 
\eqref{Sdefn}, \eqref{Adefn} and \eqref{Bdefn}.
The coefficient of $E_{\mathrm{rot}_C(\mu)}$ in $x_j E_\mu$ is 
$$(-1)^{S_{j,\mu}(C)} q^{A_{j,\mu}(C)} t^{B_{j,\mu}(C)}W_{\mu,k\in C}W_{\mu,k\not\in C}.$$
\end{prop}
\begin{proof}
By \eqref{wtmuCdefn},
$$\prod_{k\not\in C} \wt_\mu(C,k) = t^{\#\{ k\not\in C\ |\ b(k)<\mu_k\}}W_{\mu,k\not\in C}.$$
Let $i\in \{2, \ldots, m\}$ and let $k = a_i$.
If $v_\mu(a_i)<v_\mu(a_{i-1})$ then
$$
\frac{1-t}{1-q^{\mu_{a_i}-\mu_{a_{i-1}}}t^{v_\mu(a_i)-v_\mu(a_{i-1})}}
= \frac{-q^{\mu_{a_{i-1}}-\mu_{a_i}}t^{v_\mu(a_{i-1})-v_\mu(a_i)}(1-t)}
{1-q^{\vert \mu_{a_{i-1}}-\mu_{a_i}\vert}t^{\vert v_\mu(a_{i-1})-v_\mu(a_i)\vert}}
$$
and if $v_\mu(a_m)<v_\mu(a_1)$ then 
$$\frac{1}{1-q^{\mu_{a_m}+1-\mu_{a_1}}t^{v_\mu(a_m)-v_\mu(a_1)}}
=\frac{-q^{\mu_{a_1}-(\mu_{a_m}+1)}t^{v_\mu(a_1)-v_\mu(a_m)}}
{1-q^{\mu_{a_1}-(\mu_{a_m}+1)}t^{v_\mu(a_1)-v_\mu(a_m)}}
=\frac{-q^{\mu_{a_1}-(\mu_{a_m}+1)}t^{v_\mu(a_1) - v_\mu(a_m) }}
{1-q^{\vert\mu_{a_m}+1-\mu_{a_1}\vert}t^{\vert v_\mu(a_m)-v_\mu(a_1)\vert}}.
$$
These give that
$$\prod_{k\in C} \wt_\mu(C,k) = (-1)^{S'+\#\{ i \ |\ \mu_i>\mu_{a_m}\}}q^{A'}t^{B''} W_{\mu,k\in C}\cdot
\frac{1}{1-q^{\vert \gamma_p-\gamma_{p-1}\vert}t^{\vert v_\mu(a_p) - v_\mu(a_{p-1})\vert}}.
$$
So
\begin{align*}
\wt_\mu(C) 
&= t^{-\#\{ i \ |\ \mu_i>\mu_{a_m}\}}\Big(\prod_{k\in C} \wt_\mu(C,k)\Big)
\Big(\prod_{k\not\in C} \wt_\mu(C,k)\Big)
\\
&= (-1)^{S'}q^{A'}t^{B'} W_{\mu,k\in C}W_{\mu,k\not\in C}\cdot
\frac{1}{1-q^{\vert \gamma_p-\gamma_{p-1}\vert}t^{\vert v_\mu(a_p) - v_\mu(a_{p-1})\vert} }.
\end{align*}
If $p\ne 1$ then
\begin{align*}
F_\mu(C,j)
=\begin{cases}
-q^{\mu_{a_{p-1}}-\mu_{a_1}}t^{v_\mu(a_{p-1})-v_\mu(a_1)}
(1-q^{\vert \mu_{a_p}-\mu_{a_{p-1}}\vert }t^{\vert v_\mu(a_p)-v_\mu(a_{p-1})\vert}), 
&\hbox{if $\mu_{a_p}>\mu_{a_{p-1}}$,} \\
q^{\mu_{a_p}-\mu_{a_1}}t^{v_\mu(a_p)-v_\mu(a_1)}
(1-q^{\vert \mu_{a_p}-\mu_{a_{p-1}}\vert }t^{\vert v_\mu(a_p)-v_\mu(a_{p-1})\vert}), 
&\hbox{if $\mu_{a_p}<\mu_{a_{p-1}}$.} 
\end{cases}
\end{align*}
If $p=1$ then
\begin{align*}
F_\mu(C,j)
=\begin{cases}
(1-q^{\vert \mu_{a_m}+1-\mu_{a_1}\vert }t^{\vert v_\mu(a_m)-v_\mu(a_1)\vert}), 
&\hbox{if $\mu_{a_m}+1>\mu_{a_1}$,} \\
-q^{-(\mu_{a_1}-(\mu_{a_m}+1)}t^{-(v_\mu(a_1)-v_\mu(a_m)}
(1-q^{\vert \mu_{a_m}+1-\mu_{a_1}\vert }t^{\vert v_\mu(a_m)-v_\mu(a_1)\vert}), 
&\hbox{if $\mu_{a_m}+1<\mu_{a_1}$.} 
\end{cases}
\end{align*}
Thus
$$F_\mu(C,j)\wt_\mu(C) = (-1)^{S_{\mu,j}(C)}q^{A_{\mu_j}(C)}t^{B_{\mu,j}(C)}
W_{\mu,k\in C}W_{\mu,k\not\in C}$$
and the result now follows from Theorem \ref{Monkthm}(a).
\end{proof}

In order to specialize the coefficients in Proposition \ref{propIntExp} at $q=0$ and $t=0$ it is important to know that
the powers of $q$ and $t$ are nonnegative.  This is established by the following Proposition.

\begin{prop}
Let
$A_{j,\mu}(C)$ and $B_{j,\mu}(C)$ be as defined in 
\eqref{Adefn} and \eqref{Bdefn}.  Then
$$
A_{j,\mu}(C) \ge 0 \qquad\hbox{and}\qquad
B_{j,\mu}(C) \ge 0.
$$
\end{prop}
\begin{proof}
To keep track of signs, write
\begin{align*}
A_{j,\mu}(C)
&= A' + \begin{cases}
(\mu_{a_{p-1}}-\mu_{a_1}), &\hbox{if $p\ne 1$ and $\mu_{a_1}\le \mu_{a_{p-1}}\le \mu_{a_p}$,} \\
-(\mu_{a_1}-\mu_{a_{p-1}}), &\hbox{if $p\ne 1$ and $\mu_{a_1}>\mu_{a_{p-1}}\le \mu_{a_p}$,} \\
(\mu_{a_{p}}-\mu_{a_1}), &\hbox{if $p\ne 1$ and $\mu_{a_{p-1}}>\mu_{a_p}\ge a_1$,} \\
-(\mu_{a_1}-\mu_{a_p}), &\hbox{if $p\ne 1$ and $\mu_{a_{p-1}}>\mu_{a_p}< \mu_{a_1}$,} \\
-(\mu_{a_1}-(\mu_{a_m}+1)), &\hbox{if $p=1$ and $\mu_{a_m}+1\le \mu_{a_1}$,} \\
0, &\hbox{if $p=1$ and $\mu_{a_m}+1 > \mu_{a_1}$.}
\end{cases}
\end{align*}
Note that $A'\ge 0$ since it is a sum of positive integers.
Let us consider the cases when the term added to $A'$ is negative.

\smallskip\noindent
\emph{Case $\mu_{a_1}>\mu_{a_{p-1}}\le \mu_{a_p}$}: 
Since the total of the descents of the sequence $(\mu_{a_1}, \mu_{a_2}, \ldots, \mu_{a_{p-1}})$
is at least as large as $(\mu_{a_1}-\mu_{a_p})$ then
$$\sum_{i\in \{2,\ldots, p-1\}\atop \mu_{a_{i-1}}>\mu_{a_i}} (\mu_{a_{i-1}}-\mu_{a_i})\ge (\mu_{a_1}-\mu_{a_{p-1}}),
\quad\hbox{so that}\quad
\sum_{i\in \{2,\ldots, p-1\}\atop \mu_{a_{i-1}}>\mu_{a_i}} (\mu_{a_{i-1}}-\mu_{a_i})
-(\mu_{a_1}-\mu_{a_{p-1}}) \ge 0$$
and $A_{j,\mu}(C)\ge 0$.

\smallskip\noindent
\emph{Case $p\ne 1$ and $\mu_{a_{p-1}}>\mu_{a_p}<\mu_{a_1}$}:  
Since the total of the descents of the sequence $(\mu_{a_1}, \mu_{a_2}, \ldots, \mu_{a_p})$
is at least as large as $(\mu_{a_1}-\mu_{a_p})$ then
$$\sum_{i\in \{2,\ldots, p\}\atop \mu_{a_{i-1}}>\mu_{a_i}} (\mu_{a_{i-1}}-\mu_{a_i})
\ge (\mu_{a_1}-\mu_{a_p})
\qquad\hbox{so that}\qquad
A_{j,\mu}(C) = A' - (\mu_{a_1}-\mu_{a_p}) \ge 0$$
and $A_{j,\mu}(C)\ge 0$. 

\smallskip\noindent
\emph{Case $p=1$ and $\mu_{a_m}+1>\mu_{a_1}$}:  In this case the last term in the definition of $A'$
cancels with the added extra term $-(\mu_{a_1}-(\mu_{a_m}+1))$ so that $A_{j,\mu}(C)$ is a sum of positive
integers and is $\ge 0$.

A similar argument shows shows that $B_{j,\mu}(C)\ge 0$.
\end{proof}

Now we are ready to specialize the result of Proposition \ref{propIntExp} at $q=0$ and $t=0$.

\begin{cor} \label{spclzdMonk}
Keep the same notations as in Propostion \ref{propIntExp}.
\item[] \qquad \emph{(a)}  If $t=0$ then\qquad
$\displaystyle{ x_jE_\mu = \sum_{C\subseteq\{1, \ldots, n\}\atop B_{j,\mu}(C)=0} 
(-1)^{S_{j,\mu}(C)} q^{A_{j,\mu}(C)} E_{\mathrm{rot}_C(\mu)}. }$
\item[] \qquad \emph{(b)} If $q=0$ then \qquad
$$x_jE_\mu = \sum_{C\subseteq\{1, \ldots, n\}\atop A_{j,\mu}(C)=0} 
(-1)^{S_{j,\mu}(C)} t^{B_{j,\mu}(C)} (1-t)^{m-1} 
\Big(\prod_{\gamma_i=\gamma_{i-1}\atop i\ne p} \frac{1}{1-t^{v_\mu(a_i)-v_\mu(a_{i-1})}}\Big)
E_{\mathrm{rot}_C(\mu)}.
$$
\item[]\qquad \emph{(c)} If $q=0$ and $t=0$ then\qquad
$\displaystyle{ x_j E_\mu = \sum_{C\subseteq\{1,\ldots, n\}\atop A_{j,\mu}(C)=0, B_{j,\mu}(C)=0} (-1)^{S_{j,\mu}(C)} 
E_{\mathrm{rot}_C(\mu)}. }$
\end{cor}

\end{document}